\documentclass[12pt]{amsart}
\usepackage[numbers, square]{natbib}
\usepackage{amssymb}
\usepackage{amsmath}
\usepackage{amsfonts}
\usepackage{geometry}
\usepackage{graphicx}
\usepackage{amsthm}
\usepackage{hyperref}
\usepackage[dvipsnames]{xcolor}

\providecommand{\U}[1]{\protect\rule{.1in}{.1in}}
\theoremstyle{plain}

\newtheorem{corollary}{Corollary}

\newtheorem{lemma}{Lemma}

\newtheorem{proposition}{Proposition}
\newtheorem{remark}{Remark}

\newtheorem{theorem}{Theorem}
\numberwithin{equation}{section}

\DeclareMathOperator{\dist}{dist} \DeclareMathOperator{\diam}{diam}

\begin{document}
\title[Nodal sets of Robin and Neumann eigenfunctions ]
{Nodal sets of Robin and Neumann eigenfunctions}
\author{ Jiuyi Zhu}
\address{
Department of Mathematics\\
Louisiana State University\\
Baton Rouge, LA 70803, USA\\
Email:  zhu@math.lsu.edu }
\thanks{Zhu is supported in part by  NSF grant DMS-1656845 and OIA-1832961}
\date{}
%\dedicatory{}

\begin{abstract}
 We investigate the measure of nodal sets for Robin and Neumann eigenfunctions in the domain and on the boundary of the domain. A polynomial upper bound for the interior nodal sets is obtained for Robin eigenfunctions in the smooth domain. For the analytic domain, the sharp upper bounds of the interior nodal sets was shown for Robin eigenfunctions. More importantly, we obtain the sharp upper bounds for the boundary nodal sets of Neumann eigenfunctions with new quantitative global Carleman estimates.
  Furthermore, the sharp doubling inequality and vanishing order of Robin eigenfunctions on the boundary of the domain are obtained.
\end{abstract}

\maketitle
\section{Introduction}

In this paper, we consider the Robin eigenfunctions with  a possible large parameter $|\alpha|$
\begin{equation}
\left \{ \begin{array}{lll}
-\triangle u=\lambda u \quad &\mbox{in} \ {\Omega},  \medskip\\
\frac{\partial u}{\partial \nu}+\alpha u=0 \quad &\mbox{on} \ {\partial\Omega}
\end{array}
\right.
\label{robin}
\end{equation}
on a smooth and compact domain $\Omega\in \mathbb R^n$ with $n\geq 2$, where $\nu$ is an unit outer normal and $n$ is the dimension of the space. In the case of $\alpha=0$, the equations (\ref{robin}) is called the Neumann eigenvalue problem
\begin{equation}
\left \{ \begin{array}{rll}
-\triangle u=\lambda u \quad &\mbox{in} \ {\Omega},  \medskip\\
\frac{\partial u}{\partial \nu}=0 \quad &\mbox{on} \ {\partial\Omega}.
\end{array}
\right.
\label{Neumann}
\end{equation}
In the case of $\alpha=\infty$, it can be considered as the Dirichlet eigenvalue problem
\begin{equation}
\left \{ \begin{array}{rll}
-\triangle u=\lambda u \quad &\mbox{in} \ {\Omega},  \medskip\\
u=0 \quad &\mbox{on} \ {\partial\Omega}.
\end{array}
\right.
\label{Diri}
\end{equation}
For any fixed constant $\alpha$, there exists a sequence of eigenvalues $\lambda_1<\lambda_2\leq\cdots \to \infty$. If $\alpha$ is negative, the first finite number of eigenvalues can be negative. Moreover, $\lambda_k\to -\infty$ as $\alpha\to -\infty$ for any fixed $k\geq 1$.
If one considers the Robin eigenvalue problem (\ref{robin}) as an elliptic problem  in a special case that $\lambda=0$, the equation (\ref{robin}) is reduced to the Steklov eigenvalue problem
\begin{equation}
\left \{ \begin{array}{lll}
\triangle u=0 \quad &\mbox{in} \ {\Omega},  \medskip\\
\frac{\partial u}{\partial \nu}+\alpha u=0 \quad &\mbox{on} \ {\partial\Omega}.
\end{array}
\right.
\label{steklov}
\end{equation}
One may regard $-\alpha$ as the eigenvalue for the Steklov eigenvalue problem (\ref{steklov}). Hence, the model (\ref{robin}) includes general kind of eigenvalue problems for Laplacian.

We are interested in the nodal sets of eigenfunctions in (\ref{robin}) and (\ref{Neumann}) in the domain $\Omega$ and on the boundary $\partial\Omega$.
The nodal sets are the zero level sets of eigenfunction. For the eigenfunctions of Laplacian
\begin{equation}
\triangle u+\lambda u=0
\end{equation}
on a compact smooth Riemannian manifold $\mathcal{M}$, Yau \cite{Y} conjectured that the Hausdorff measure of nodal sets can be controlled above and below by eigenvalues as
$$ c\sqrt{\lambda}\leq H^{n-1}( \{x\in\mathcal{M}| u(x)=0\})\leq C\sqrt{\lambda},   $$
where $c$, $C$ depend on the manifold $\mathcal{M}$. For the real analytic manifolds, the conjecture was answered by Donnelly-Fefferman in their seminal paper \cite{DF}.
A relatively simpler proof for the upper
bound of general second order elliptic equations on the analytic domain was given by  Lin \cite{Lin} by a different approach.

For the smooth manifolds with $n=2$, Donnelly-Fefferman \cite{DF1} and  Dong \cite{D} independently showed the upper bound  $H^{1}(\{x\in\mathcal{M}| u(x)=0\})\leq C \lambda^{\frac{3}{4}} $ by using different arguments.
 A slight improvement with upper bound $C \lambda^{\frac{3}{4}-\epsilon}$ was given by  Logunov and
Malinnikova \cite{LM}.  For higher dimensions $n\geq 3$, Hardt and Simon \cite{HS} derived the exponential upper bound
$H^{n-1}(\{\mathcal{M} | u=0\})\leq C \lambda^{C\sqrt{\lambda}}$.  Very recently,
 Logunov in \cite{Lo} obtained a polynomial upper bound
\begin{equation} H^{n-1}( \{x\in\mathcal{M} | u(x)=0\})\leq C \lambda^{\beta},
\label{poly}
\end{equation}
where $\beta>\frac{1}{2}$ depends only on the dimension.
For the lower bound, Logunov \cite{Lo1} completely answered the Yau's conjecture and obtained the sharp lower bound as
$ c\sqrt{\lambda}\leq H^{n-1}(\{x\in\mathcal{M}| u(x)=0\})$ for
 smooth manifolds in any dimensions. For $n=2$, such sharp lower bound was obtained earlier by Br\"uning \cite{Br}. This breakthrough improved a polynomial lower bound obtained early by Colding and  Minicozzi \cite{CM}, Sogge and Zelditch \cite{SZ}. See also other polynomial lower bounds by different methods, e.g. \cite{HSo}, \cite{M}, \cite{S} and other related results on nodal sets of eigenfunctions, e.g. \cite{Han}, \cite{HLu}, \cite{Lu}.
 The recent breakthrough on nodal sets of eigenfunctions in \cite{LM},  \cite{Lo} and \cite{Lo1} is based on seminal work on new combinatorial arguments for doubling index and further exploration of frequency functions in \cite{GL} and \cite{HL}.

For the Neumann eigenvalue problem (\ref{Neumann}) or the Dirichlet eigenvalue problem (\ref{Diri}), the polynomial upper as (\ref{poly}) can be derived
\begin{equation} H^{n-1}(\{x\in\Omega | u(x)=0\})\leq C \lambda^{\beta}
\label{poly2}
\end{equation}
for smooth domains with some $\beta>\frac{1}{2}$ depending only on the dimension. One can construct a double manifold $\tilde{\Omega}=\Omega\cup\Omega$ to get rid of the boundary. Then one can do an even extension for the Neumann eigenvalue problem  or an odd extension for the  Dirichlet eigenvalue problem on the domain to have second order elliptic equations with Lipschitz metrics. The following sharp doubling inequality on the double manifold
\begin{equation}
\|u\|_{L^2(\mathbb B_{2r}(x))}\leq e^{ C\sqrt{\lambda}}\|u \|_{L^2(\mathbb B_{r}(x))}
\label{doubleuse}
\end{equation}
can be deduced as \cite{DF2} for the second order elliptic equations with Lipschitz coefficients. Applying the new combinatorial arguments in \cite{Lo} for the aforementioned second order elliptic equations with Lipschitz coefficients and doubling inequality (\ref{doubleuse}), one can obtain the polynomial upper bound (\ref{poly2}). We are interested in the measure of nodal sets in $\Omega$ for general eigenvalue problems of Laplacian (\ref{robin}). Especially, we want to find out how the upper bound of nodal sets depends on possible large parameter $|\alpha|$ on the boundary. For the interior nodal sets, we can show that

\begin{theorem}
Let $u$ be the Robin eigenfunction in (\ref{robin}) with $n\geq 3$. There exists a positive constant $C$ depending only on the smooth domain $\Omega$ such that
\begin{equation} H^{n-1}(\{x\in \Omega |u(x)=0\})\leq C(|\alpha|+\sqrt{|\lambda|})^{\beta},
\label{interior} \end{equation}
where $\beta>1$ depending only on the dimension $n$.
\label{th1}
\end{theorem}

We briefly sketch the proof of the theorem. To prove (\ref{interior}), we first need to derive the sharp doubling inequality
\begin{equation}
\|u\|_{L^2(\mathbb B_{2r}(x))}\leq e^{ C(|\alpha|+\sqrt{|\lambda|})}\|u \|_{L^2(\mathbb B_{r}(x))}
\label{doubleuse2}
\end{equation}
on the double manifold $\tilde{\Omega}$. We introduce an auxiliary function involving the distance to the boundary, then transform the Robin eigenvalue problem into  second order elliptic equations with Neumann boundary conditions. We do an even reflection and obtain some quantitative Carleman estimates to show (\ref{doubleuse2}) on the double manifold. The combination of  the results in \cite{Lo} and the doubling inequality (\ref{doubleuse2}) implies Theorem \ref{th1}.

For the interior nodal sets of Robin eigenfunctions in analytic domains, we can show that
\begin{theorem}
Let $u$ be the Robin eigenfunction in (\ref{robin}). There exists a positive constant $C$ depending only on the real analytic domain $\Omega$ such that
\begin{equation} H^{n-1}(\{x\in \Omega |u(x)=0\})\leq C(|\alpha|+\sqrt{|\lambda|}).
\label{interior2} \end{equation}
\label{th4}
\end{theorem}

The interior nodal sets estimates for Dirichlet eigenvalue problem (\ref{Diri}) and Neumann eigenvalue problem (\ref{Neumann}) in real analytic domains have been shown by Donnelly and Fefferman in \cite{DF2} to be
$$ H^{n-1}(\{x\in \Omega |u(x)=0\})\leq C \sqrt{\lambda}.$$
Our strategy is to use the doubling inequality (\ref{doubleuse2}) and a growth control lemma on the number of zeros for complex analytic functions. We first find out the upper bound for nodal sets as (\ref{interior2}) for the regions in the neighborhood of the boundary, then obtain the nodal sets estimates for regions away from the boundary. The combination of the estimates in the two regions gives Theorem 2.

The rest of the paper is devoted to a challenging topic on the measure of nodal sets on the boundary. There are few results on the study of the measure of boundary nodal sets.
For the Neumann and Robin eigenfunctions, it is possible that the nodal sets of eigenfunctions in $\Omega$ intersect the boundary $\partial\Omega$. Thus, it is interesting to find out how large the measure of boundary nodal sets is and how the measure depends on $\alpha$ and $\lambda$. The nodal sets on the boundary are co-dimension one. For the Neumann eigenfunctions, we can show that
\begin{theorem}
Let $u$ be the Neumann  eigenfunction in (\ref{Neumann}) in the real analytic domain $\Omega$. There exists a positive constant $C$ that depends only on the domain $\Omega$ such that
\begin{align}
 H^{n-2}(\{x\in \partial\Omega| u(x)=0\})\leq C\sqrt{\lambda}.
 \label{Sha}
 \end{align}
\label{thnew}
\end{theorem}

The upper bound of nodal sets on the boundary for Neumann eigenfunctions in the Theorem is optimal. Such upper bound for boundary nodal sets of Neumann eigenfunctions was derived by Toth and Zelditch \cite{TZ} in planar analytic domains using a different method. Theorem \ref{thnew} improves such result to general dimensions. The discussion of the sharpness is provided in Remark 4 in Section 5.

Using the similar idea, we further study the upper bounds for the boundary nodal sets of Robin eigenfunctions.
\begin{corollary}
Let $u$ be the Robin eigenfunction in (\ref{robin}) in the real analytic domain $\Omega$.
There exists a positive constant $C$ depending only on the domain $\Omega$  such that
$$ H^{n-2}(\{x\in\partial\Omega | u(x)=0\})\leq C(|\alpha|+\sqrt{|\lambda|}). $$
\label{th2}
\end{corollary}

Since the model (\ref{robin}) includes the Steklov eigenvalue problem (\ref{steklov}) as the special case with $\lambda=0$, the upper bound in the corollary includes the sharp results for boundary nodal sets obtained by Zeldtich \cite{Z} for Steklov eigenfunctions.
To derive the results in Theorem  \ref{thnew} and the Corollary \ref{th2}, much more efforts are devoted to obtaining the doubling inequality for Robin eigenfunctions on the boundary $\partial\Omega$. It is also challenging to obtain doubling inequalities on the boundary.
\begin{theorem}
Let $u$ be the Robin eigenfunction in (\ref{robin}). There exist positive constants $C$ and $r_0$ depending only on the smooth domain $\Omega$ such that
\begin{equation}
\|u\|_{L^2(\mathbb B_{2r}(x))}\leq e^{C(|\alpha|+\sqrt{|\lambda|})}\|u\|_{L^2(\mathbb B_{r}(x))}
\label{LLL}
\end{equation}
for any $0<r<r_0$ and any $\mathbb B_{2r}(x)\subset\partial\Omega $.
\label{th3}
\end{theorem}

To obtain (\ref{LLL}), we prove a new quantitative propagation of smallness lemma (i.e. Lemma \ref{halfspace}) with possible large $|\alpha|$ or $|\lambda|$, which is based on a new and novel global quantitative Carleman estimates with boundary terms (i.e. Proposition \ref{Pro4}). A direct consequence of Theorem \ref{th3} is the following vanishing order estimates.
\begin{corollary}
Let $u$ be the Robin eigenfunction in (\ref{robin}). Then the vanishing order of solution $u$ on $\partial\Omega$ is everywhere less than
$C(|\alpha|+\sqrt{|\lambda|})$, where $C$ depends only the smooth domain $\Omega$.
\label{cor2}
\end{corollary}

Let us give some comments on those aforementioned results.
\begin{remark}
The results in Theorem 1, 2, 4  actually hold for either $|\alpha|$ or $|\lambda|$ large. The results in Theorem 2, 4 are sharp, which can be observed from the balls. The general eigenvalue problem (\ref{robin}) includes Steklov eigenvalue problem as the special case. For the study of nodal sets and doubling estimates of Steklov eigenfunctions, see e.g. \cite{BL}, \cite{Z}, \cite{WZ}, \cite{SWZ}, \cite{Zh}, \cite{Zh1},  \cite{PST}, \cite{Zh2}, \cite{GR},  etc. Steklov eigenfunctions can be regarded as eigenfunctions of the Dirichlet-to-Neumann map on the boundary. Thus, global Fourier analysis techniques can be applied. However, those arguments seem not be used for the eigenvalue problem (\ref{robin}). Some new and novel global Carleman estimates are developed to obtain boundary doubling inequalities and boundary nodal sets.
The conclusions in Theorem \ref{th1} and 4 also hold for Robin eigenfunctions of Laplace-Beltrami operator on any smooth and compact Riemannian manifolds. The results in Theorem 2 and 3 are true for Neumann and Robin eigenfunctions on real analytic compact Riemannian manifolds.
\label{rre2}
\end{remark}

\begin{remark}
For Robin eigenvalue problems, the eigenvalue $\lambda$ depends on the parameter $\alpha$. It is interesting to study the asymptotic estimates of $\lambda$ with respect to $\alpha$. If $\alpha<0$, it has been shown in \cite{DK} that
$$ \lim_{\alpha\to -\infty}\frac{\lambda_k}{-\alpha^2}=1 $$
for every $k\geq 1$. Thus, the eigenvalue $|\lambda_k|$ and $\alpha^2$ grow at the same rate. In this case with $|\alpha|$ sufficiently large, we can replace the term $|\alpha|$ by $\sqrt{|\lambda|}$ in Theorem 1, 2, 4 and Corollary 1, 2.
\end{remark}

The organization of the article is as follows. Section
 2 is devoted to the transformation of the Robin eigenvalue problem to elliptic equations with the Neumann boundary conditions. The polar coordinates for the double manifold with Lipschitz metrics is also constructed. In section 3, using the local quantitative Carleman estimates, we establish
 some quantitative three-ball theorems.  Then we derive the doubling inequality on the double manifold, the polynomial growth of nodal sets for Robin eigenfunctions on smooth domains and the doubling inequality on the half balls. Section 4 is used to show the upper bounds for interior nodal sets for the Robin eigenfunction on real analytic domains. Section 5 is devoted to the boundary doubling inequality and nodal sets estimates on the real analytic boundary. In the section 6, we derive a new type of global quantitative Carleman estimates with boundary terms. In the appendix, we include the proof of three-ball theorem presented in section 3.
The letters $C$ and $C_i$ denote generic positive
constants that do not depend on $u$, and may vary from line to
line. In the paper, since we study the asymptotic properties for eigenfunctions, we assume that either $|\alpha|$ or $|\lambda|$ is sufficiently large.

\section{Preliminary }
In this section, we transform the Robin eigenvalue problem to elliptic equations with Neumann boundary conditions.
We want to move the  parameter $\alpha$ on the boundary into the coefficients in a second order elliptic equation. At first, we will transform the Robin eigenvalue problem to a Neumann boundary problem. Considering a small $\rho$-neighborhood of smooth $\partial\Omega$, let $$\Omega_\rho=\{x\in\Omega| \dist(x, \ \partial\Omega)<\rho\},$$ where $\dist(x, \ \partial\Omega)=d(x)$ is the distance function to the boundary $\partial\Omega$. Since the domain $\Omega$ is smooth, there exists some small $\rho_0$ depending only on $\Omega$ such that the distance function $d(x)\in C^\infty$ in $\Omega_\rho$ for $0<\rho<\rho_0$. If $x\in\partial \Omega$, it is known that
\begin{equation} \nabla d(x)=-\nu(x),
\end{equation} where $\nu(x)$ is an unit outer normal at $x$. Inspired by the construction in \cite{BL} for Steklov eigenfunctions, we introduce the following auxiliary function
\begin{equation}
\bar u(x)=e^{-\alpha d(x)} u(x) \quad \quad \mbox{for} \ x\in \Omega_\rho\cup\partial\Omega.
\label{baru}
\end{equation}

It is easy to check that $\bar u(x)$ satisfies the following second order elliptic equations in a neighborhood of $\Omega$
\begin{equation}
\left \{ \begin{array}{rll}
\triangle \bar u+2\alpha \nabla d(x)\cdot \nabla \bar u+ (\alpha \triangle d(x)+\alpha^2 |\nabla d(x)|^2  +\lambda) \bar u&=0 \quad \mbox{in} \ \Omega_\rho,   \medskip\\
\frac{\partial \bar u}{\partial \nu}&=0 \quad \mbox{on} \ \partial\Omega.
\end{array}
\right.
\label{trans}
\end{equation}

We use Fermi coordinates near the boundary to flatten the boundary. Let $0\in \partial\Omega$.  We can find a small constant $\rho>0$ so that there exists a map $(x', \ x_n)\in \partial \Omega\times [0, \ \rho)\to \Omega$ sending $(x', \ x_n)$ to the endpoint, $x\in \Omega$, with length $x_n$,  which starts at $x'\in \partial \Omega$ and is perpendicular to $\partial\Omega$. Such map is a local diffeomorphism. Note that $d(x)=x_n$ in the coordinates and $x'$ is the geodesic normal coordinates of $\partial\Omega$. The metric takes the form
$$ \sum^n_{i,j=1} g_{ij} dx^i dx^j= dx_n^2+ \sum^{n-1}_{i,j=1} g'_{ij}(x', x_n) dx^i dx^j,   $$
where  $g'_{ij}(x', x_n)$ is a Riemannian metric on $\partial \Omega$ depending smooth on $x_n\in [0, \ \rho).$ In a neighborhood of the boundary, the Laplace can be written as
\begin{align} \triangle =\sum^n_{i,j=1} g^{ij}\frac{\partial^2}{\partial x_i \partial x_j}+\sum^n_{i=1} q_i(x)\frac{\partial}{\partial x_i}
\label{newlap}
\end{align}
using local coordinates for $\partial \Omega$, where  $g^{ij}$ is the matrix with entries $(g^{ij})_{1<i\leq j<n-1}= (g'_{ij})^{-1}$ and $g^{nn}=1$ and $g^{nk}=g^{kn}=0$ for $k\not =n$, and $ q_i(x)\in C^\infty$.

In the local coordinates, we identify $\partial\Omega$ locally as $\{x_n=0\}$. The Fermi distance function from $0$ on a relatively open neighborhood $0$ in $\Omega$ is defined by
$$ \tilde{r}=\sqrt{x_1^2+\cdots+ x_{n-1}^2+ x_n^2}.$$
The Fermi exponential map at $0$, ${exp}_0$, which gives the Fermi coordinate system, is defined on a half space of $\mathbb R^n_+$.
 We choose a Fermi half-ball $\tilde{\mathbb B}^+_\delta(0)$ centered at origin at $\{x_n=0\}$ for $0<\delta<10 \rho_0$.  It is known that $\mathbb B_{\delta/2}(0)\cap \Omega \subset \tilde{\mathbb B}^+_\delta(0)\subset  {\mathbb B}_{2\delta}(0)\cap \Omega $, where $\mathbb B_\delta(0)$ is the ball centered at origin with radius $\delta$ in the Euclidean space. See e.g. the appendix A in \cite{LZ}.
For  ease of notation, we still write $\tilde{\mathbb B}^+_\delta(0)$ as $\mathbb B^+_{\delta}(0)$ . Then it follows from (\ref{trans}) that $\bar u$ satisfies the following equation in a neighborhood of the boundary
\begin{equation}
\left \{ \begin{array}{rll}
\triangle_g \bar u+ \bar b(x)\cdot \nabla \bar u+ \bar c(x) \bar u&=0 \quad \mbox{in} \ \mathbb B^+_\delta(0), \medskip \\
\frac{\partial \bar u}{\partial \nu}&=0 \quad \mbox{on} \ \{x_n=0\},
\end{array}
\right.
\label{half}
\end{equation}
where $g=(g_{ij})_{n\times n}$  is smooth in $\mathbb B_\delta^+$, and $\bar b(x)$ and $\bar c(x)$ satisfy
\begin{equation}
\left\{ \begin{array}{lll}
\|\bar {b}\|_{C^\infty(\mathbb B^+_\delta)}\leq C(|\alpha|+1), \medskip \\
\|\bar {c}\|_{C^\infty(\mathbb B^+_\delta)}\leq C(\alpha^2+|\lambda|)
\end{array}
\right.
\label{halfcon}
\end{equation}
with $C$ depending only on $\partial \Omega$.

 We also want to consider the eigenfunction globally on $\Omega$. As it is discussed that the distance function $d(x)=dist(x, \ \partial\Omega)$ is  smooth to the boundary $\partial\Omega$ in a small neighborhood $\Omega_\rho$ for some small $\rho$,  we make a smooth extension for $d(x)$ in the whole $\Omega$. Then we introduce a smooth function $l(x)$ such that $\varrho(x)$ defined as
\begin{equation}
\varrho(x)=\left \{\begin{array}{lll} d(x) \quad x\in {\Omega}_{\rho}, \medskip\nonumber\\
l(x) \quad x\in  \Omega\backslash\Omega_{\rho}
\end{array}
\right. \label{errorr}
\end{equation}
is a smooth function in the whole $\Omega$. Performing the similar procedure as before, we first transform the Robin eigenvalue problem to a Neumann boundary problem. Let
\begin{equation}
\bar u(x)=e^{-\alpha \varrho(x)}  u(x) \quad \quad \mbox{for} \ x\in \Omega.
\label{441}
\end{equation}
Then $\bar u(x)$ satisfies the following  Neumann boundary problem
\begin{equation}
\left \{ \begin{array}{rll}
\triangle \bar u+2\alpha \nabla \varrho(x)\cdot \nabla \bar u+ (\alpha \triangle \varrho(x)+\alpha^2 |\nabla \varrho(x)|^2  +\lambda) \bar u&=0, \quad \mbox{in} \ \Omega, \medskip\\
\frac{\partial \bar u}{\partial \nu}&=0 \quad \mbox{on} \ \partial\Omega.
\end{array}
\right.
\label{iff}
\end{equation}

We want to get rid of the boundary $\partial\Omega$ as well. We define a global double manifold $\tilde{\Omega}=\Omega\cup \Omega$. To extend $\bar u$ to be on the double manifold $\tilde{\Omega}$, we consider an even extension, that is $$\bar u\circ \pi= \bar u,$$ where $\pi: \tilde{\Omega}\to \tilde{\Omega}$ is a cononical involutive isometry which interchanges the two copies of  $\tilde{\Omega}$. Near the boundary $\partial \Omega$, the new metric $\tilde{g}$ on the double manifold $\tilde{\Omega}$ is Lipschitz continuous. To explain the metric $\tilde{g}$ is only Lipschitz near the boundary, we use Fermi coordinates with respect to the boundary as before.
 The differential structure of $\tilde{\Omega}$ near $\partial\Omega$ uses the Fermi coordinates in $g_{ij}$. So $x_n>0$ and $x_n<0$ define the two copies of $\Omega$. In these coordinates, $g^{nk}=0$ for $k\not =n$, there are no cross terms between $\partial_n$ and $\partial_{x_i}$. The metric $g_{ij}(x', |x_n|)$ is symmetric under $x_n\to -x_n$. Thus, it is Lipschitz continuous across $\partial \Omega$. Under the new metric $\tilde{g}$ on the double manifold, from the equations (\ref{iff}),  the new solution $\bar u$ satisfies  second order elliptic equations
\begin{align}
\triangle_{\tilde{g}} \bar u+ \tilde{b}(x) \cdot \nabla \bar u+ \tilde{c}(x)  \bar u&=0 \quad \mbox{in} \ \tilde{\Omega},
\label{back}
\end{align}
where $\tilde{b}$ and $\tilde{c}$ satisfy
\begin{equation}
\left\{ \begin{array}{lll}
\|\tilde{b}\|_{W^{1,\infty}}\leq C(|\alpha|+1), \medskip \\
\|\tilde{c}\|_{W^{1,\infty}}\leq C(\alpha^2+|\lambda|).
\end{array}
\right.
\label{back1}
\end{equation}
%%where we have ignored the only continuous parts in $\tilde{b}$ that do not depend on $\alpha$ or $\lambda$, since they do not play the role in the quantitative estimates for doubling inequality and nodal sets in terms of $\alpha$ or $\lambda$ in the future arguments.

Now we deal with the second order elliptic equations with Lipschitz continuous coefficients. In order to apply Carleman estimates, we want to use polar coordinates. Following the strategy on the regularization for Lipschitz metric in \cite{AKS} by Aronszajn, Krzywicki and Szarski, we are still able to introduce a suitable geodesic normal coordinates. Readers who are familiar with such construction may skip the rest of the section. Without loss of generality, we consider the construction of normal coordinates at origin. Starting from a ball $\mathbb B_\delta$ in local coordinates, we introduce a ``radial" coordinate and a conformal change metric $\hat{g}_{ij}$. Let
 \begin{equation}
 r=r(x)=(\tilde{g}_{ij}(0)x_i x_j)^{\frac{1}{2}}
 \label{radial}
 \end{equation}
 and
\begin{equation}
\hat{g}_{ij}(x)= \tilde{g}_{ij}(x){\hat{\psi}}(x),
\end{equation}
where \begin{equation} {\hat{\psi}}(x)= \tilde{g}^{kl}(x)\frac{\partial r}{\partial x^k}\frac{\partial r}{\partial x^l}\end{equation} for $x\not=0$ and  $(\tilde{g}^{ij})=(\tilde{g}_{ij})^{-1}$ is the inverse matrix. In the whole paper, we adopt the Einstein notation. The summation over index is understood.
We assume the uniform ellipticity condition holds in $\mathbb B_\delta$ for
$$ \Lambda_1\|\xi\|^2 \leq \sum^n_{i,j=1} \tilde{g}_{ij}(x)\xi_i\xi_j\leq  \Lambda_2\|\xi\|^2 $$
for some positive constant $\Lambda_1$ and $\Lambda_2$ depending only on $\Omega$. Then $\hat{\psi}$ is bounded above and below satisfying
\begin{equation}
\frac{\Lambda_1}{\Lambda_2}\leq \hat{\psi}\leq \frac{\Lambda_2}{\Lambda_1}.
\end{equation}
We can also see that $\hat{\psi}$ is Lipschitz continuous. With these auxiliary quantities, the following replacement of geodesic polar coordinates are constructed in \cite{AKS}.
In the geodesic ball
$\hat{ \mathbb B}_{\hat {r}_0}=\{ x\in  \tilde{\Omega}| r(x)\leq \hat{r}_0\},$ the following properties hold: \medskip \\
(i) $\hat{g}_{ij}(x) $ is Lipschitz continuous;\medskip\\
(ii) $\hat{g}_{ij}(x) $ is uniformly elliptic with $ \frac{\Lambda_1^2}{ \Lambda_2 }\|\xi\|^2\leq \hat{g}_{ij}(x)\xi_i\xi_j\leq \frac{\Lambda_2^2}{ \Lambda_1}\|\xi\|^2.  $ \medskip\\
(iii) Let $\Sigma =\partial \hat{ \mathbb B}_{\hat {r}_0}$. We can parametrize $\hat{ \mathbb B}_{\hat {r}_0} \backslash \{0\}$ by the polar coordinate $r$ and $\theta$, with $r$ defined by (\ref{radial}) and $\theta=(\theta_1, \cdots \theta_{n-1})$ be the local coordinates on $\Sigma$. In these polar coordinates, the metric can be written as
\begin{equation}
\hat{g}_{ij}(x)  dx^i dx^j= dr^2+ r^2 \hat{\gamma}_{ij} d\theta^i d\theta^j
\end{equation}
with $\hat{\gamma}_{ij}=\frac{1}{r^2} \hat{g}_{kl}(x) \frac{\partial x^k}{\partial \theta^i}\frac{\partial x^l}{\partial \theta^j}$. \medskip\\
(iv) There exists a positive constant $M$ depending on $\tilde{g}_{ij}$ such that for any tangent vector $\xi_j\in T_{\theta}(\Sigma)$,
\begin{align}
|\frac{\partial \hat{\gamma}_{ij}(r, \theta)}{\partial r} \xi^i \xi^j|\leq M| \hat{\gamma}_{ij}(r, \theta)\xi^i \xi^j|.
\label{spheress}
\end{align}
Let $\hat{\gamma}=\det{ (\hat{\gamma}_{ij})}$. Then (\ref{spheress}) implies that
\begin{equation}
|\frac{\partial \ln \sqrt{\hat{\gamma}}}{\partial r}| \leq CM.
\end{equation}

The existence of the coordinates $(r, \ \theta)$ allows us to pass to ``geodesic polar coordinates".
In particular, $r(x)=(\tilde{g}_{ij}(0)x_i x_j)^{\frac{1}{2}}$ is the geodesic distance to the origin in the metric $\hat{g}_{ij}$.
In the new metric $\hat{g}_{ij}$, the  Laplace-Beltrami operator is $$\triangle_{\hat{g}}=\frac{1}{\sqrt{\hat{g} }} \frac{\partial}{\partial x_i}( \hat{g}^{ij}\sqrt{ \hat{g}} \frac{\partial}{\partial x_j}   ),$$ where $\hat{g}= \det(\hat{g}_{ij})$.  If $\bar u$ is a solution of (\ref{back}), then $\bar u$ is locally the solution of the equation
\begin{equation}
\triangle_{\hat{g}} \bar u +\hat{b}(x)\cdot\nabla \bar u+ \hat{c}(x) \bar u=0 \quad \mbox{in} \ \hat{\mathbb B}_{\hat{r}_0},
\label{target}
\end{equation}
where
\begin{equation}
\left\{ \begin{array}{lll}
&\hat{b}_i=\frac{2-n}{2\hat{\psi}^2} \tilde{g}^{ij} \frac{\partial \hat{\psi} }{\partial x_j}+\frac{1}{\hat{\psi}}\tilde{b}_i,   \medskip \\
&\hat{c}(x)=\frac{\tilde{c}(x)}{\hat{\psi}}.
\end{array}
\right.
\end{equation}
By the properties of $\hat{\psi}$, we can see $\hat{c}(x)$ is Lipschitz continuous. Since the term $\frac{2-n}{2\hat{\psi}^2} \tilde{g}^{ij} \frac{\partial \hat{\psi} }{\partial x_j}$ in $\hat{b}_i$ is only continuous and does not depend on either $\alpha$ or $\lambda$, it can be ignored in the future quantitative estimates for doubling inequality or nodal sets. The major term $\frac{1}{\hat{\psi}}\tilde{b}_i$ is Lipschitz continuous. From the conditions in (\ref{back1}), we still write the conditions for $\hat{b}$ and $\hat{c}$ as
\begin{equation}
\left\{ \begin{array}{lll}
\|\hat {b}\|_{W^{1,\infty}(\hat{\mathbb B}_{\hat{r}_0})}\leq C (|\alpha|+1), \medskip \\
\|\hat { c}\|_{W^{1,\infty}(\hat{\mathbb B}_{\hat{r}_0})}\leq C (\alpha^2+|\lambda|).
\end{array}
\right.
\label{targetcon}
\end{equation}
For simplicity, we may write $\triangle_{\hat{g}}$ or $\triangle_{{g}}$ as $\triangle$ if the metric is understood. Since the geodesic balls or half balls under different metrics are comparable, we write all as $\mathbb B_r(x)$ or $\mathbb B^+_r(x)$ centered at $x$ with radius $r$.

\section{ Interior doubling inequality and interior nodal sets on smooth domains}

Let $r=r(y)$ be the Riemannian distance from origin to $y$.
Our major tools to get the three-ball theorem and doubling inequality are the quantitative
 Carleman estimates. Carleman estimates are weighted integral inequalities with a weight function $e^{\tau\psi}$, where $\psi$ usually satisfies some convex condition. We construct the weight function $\psi$ as follows.
 Set $$\psi(y)=-g(\ln r(y)),$$ where $g(t)=t+\log t^2$ for $-\infty<t <T_0$, and $T_0$ is negative with $|T_0|$ large enough. One can check that
 \begin{equation}
 \lim_{t \to -\infty}-e^{-t} g''(t)=\infty \quad \mbox{and} \quad \lim_{t \to -\infty} g'(t)=1.
 \end{equation}
Define
\begin{equation}
\psi_\tau(y)=e^{\tau \psi(y)}.
\end{equation}
We state the following quantitative Carleman estimates. The similar Carleman estimates with lower bound of the parameter $\tau$ has been obtained in e.g. \cite{DF}, \cite{BC}, \cite{Zh2}. Interested readers may refer to them for the proof of the following proposition.
\begin{proposition}
There exist positive constants $C_1$, $C_0$ and small $r_0$, such that for $v\in C^{\infty}_{0}(
\mathbb B_{r_0} \backslash \mathbb B_{\rho}),$  and
$$\tau>C_1(1+|\alpha|+\sqrt{|\lambda|}),$$
one has
\begin{align}
C_0\|r^2 \psi_\tau\big(\triangle v +\hat{b}(y)\cdot\nabla v+ \hat{c}(y) v\big)\|^2&\geq  \tau^3\|\psi_\tau {(\log
r)}^{-1}
u \|^2
+\tau\|r \psi_\tau {(\log r)}^{-1}\nabla v \|^2\nonumber \medskip \\ &+ \tau \rho
\|r^{-\frac{1}{2}}\psi_\tau
 v \|^2.
\label{carl} \end{align} \label{pro1}
\end{proposition}

The $\|\cdot\|_r$ or $\|\cdot\|$ norm in the whole paper denotes the $L^2$ norm over $\mathbb B_r(0)$ if
not explicitly stated.  Specifically, $\|\cdot\|_{\mathbb B_r(y))}$  for short denotes the $L^2$ norm on the ball $\mathbb B_r(y)$. Thanks to the quantitative Carleman estimates, it is a standard way to derive a quantitative three-ball theorem. Let
 $\bar u$ be the solutions of the second order elliptic equations (\ref{target}). We apply such Carleman estimates  with $v=\eta \bar u $, where $\eta$ is an appropriate smooth cut-off function, and then select an appropriate choice of the parameter $\tau$.  The statement of the quantitative three-ball theorem is  as follows.

\begin{lemma}
There exist positive constants $\bar r_0$, $C$ which depend only on $\Omega$ and $0<\beta<1$  such that, for any $0<R<\bar r_0$,  the solutions $\bar u$ of (\ref{target}) satisfy
\begin{equation}
\|\bar u \|_{\mathbb B_{2R}(x_0)}\leq e^{C(|\alpha|+\sqrt{|\lambda|})} \|\bar u \|^{\beta}_{\mathbb B_R(x_0)} \|\bar u  \|^{1-\beta}_{\mathbb B_{3R}(x_0)}
\label{quanti}
\end{equation}
\label{lemm1}
for any $x_0\in\tilde{\Omega}$.
\end{lemma}

Since the proof of three-ball theorem is kind of standard by the applications of quantitative Carleman estimates (\ref{carl}). For the completeness of the presentation, we leave the proof in the appendix.

Let $\|u\|_{L^2({\Omega})}=1$. Because of the even extension, we may write $$\|\bar u\|_{L^2(\tilde{\Omega})}=2.$$ Set $\bar x$ be the point where
$$ \|\bar u\|_{L^2(\mathbb B_{\hat{r}_0}(\bar x))}=\max_{x\in \tilde{\Omega}} \|\bar u\|_{L^2(\mathbb B_{\hat{r}_0}(x))}  $$
for some $0<\hat{r}_0< \frac{\bar r_0}{8}$.
The compactness of $\tilde{\Omega}$ implies that  $$\|\bar u\|_{L^2(\mathbb B_{\hat{r}_0}(\bar x))}\geq C_{\hat{r}_0}$$ for some $C_{\hat{r}_0}$ depending on $\tilde{\Omega}$ and $\hat{r}_0$. From the quantitative three-ball inequality (\ref{quanti}), at any point $x\in \Omega$, one has
\begin{equation}
 \|\bar u\|_{L^2(\mathbb B_{{\hat{r}_0}/2}(x))}\geq e^{-C(|\alpha|+\sqrt{|\lambda|})} \|\bar u\|^{\frac{1}{\beta}}_{L^2(\mathbb B_{{\hat{r}_0}}(x))}.
\label{use}
\end{equation}
Let $l$ be a geodesic curve between $\hat{x}$ and $\bar x$, where $\hat{x}$ is any point in $\tilde{\Omega}$. Define $x_0=\hat{x}, \cdots, x_m=\bar x$ such that $x_i\in l$ and $\mathbb B_{\frac{\hat{r}_0}{2}}(x_{i+1})\subset \mathbb B_{{\hat{r}_0}}(x_{i})$ for $i$ from $0$ to $m-1$. The number of $m$  depends only on $\diam (\tilde{\Omega})$ and $\hat{r}_0$. The properties of $(x_i)_{1\leq i\leq m}$ and the inequality (\ref{use}) imply that
\begin{align}
 \|\bar u\|_{L^2(\mathbb B_{{\hat{r}_0}/2}(x_i))}\geq e^{-C(|\alpha|+\sqrt{|\lambda|})} \|\bar u\|^{\frac{1}{\beta}}_{L^2(\mathbb B_{{\hat{r}_0}/2}(x_{i+1}))}.
\end{align}
Iterating the argument to get to $\bar x$, we obtain that
\begin{align}
 \|\bar u\|_{L^2(\mathbb B_{{\hat{r}_0}/2}(\hat{x}))}&\geq e^{-C_{\hat{r}_0}(|\alpha|+\sqrt{|\lambda|})} C_{\hat{r}_0}^{\frac{1}{\beta^m}} \nonumber \\
 &\geq e^{-C_{\hat{r}_0}(|\alpha|+\sqrt{|\lambda|})} \|\bar u\|_{L^2(\tilde{\Omega})}.
 \label{useful}
\end{align}
Let $A_{R, \ 2R}=(\mathbb B_{2R}(x_0)\backslash \mathbb B_{R}(x_0))$ for any $x_0\in \tilde{\Omega}$. Then there exists $\mathbb B_{{\hat{r}_0}/2}(\hat{x})\subset A_{ \hat{r}_0, \ 2\hat{r}_0}$ for some $\hat{x}\in  A_{2\hat{r}_0, \  \hat{r}_0}$. Thus, by (\ref{useful}),
\begin{equation}
\| \bar u\|_{L^2(A_{\hat{r}_0, \ 2\hat{r}_0} )}\geq e^{-C_{\hat{r}_0}(|\alpha|+\sqrt{|\lambda|})} \| \bar u\|_{L^2(\tilde{{\Omega}})}.
\label{annu1}
\end{equation}

With aid of the quantitative Carleman estimates (\ref{carl}) and the inequality (\ref{annu1}), using the argument as the proof of Lemma \ref{lemm1},  we are ready to derive the doubling inequality as follows.
\begin{proposition}
Let $\bar u$ be the solution of (\ref{target}) satisfying the condition (\ref{targetcon}). There exists a positive constant $C$ depending only on $\tilde{\Omega}$ such that the doubling inequality holds
\begin{equation}
\|\bar u\|_{L^2(\mathbb B_{2r}(x))}\leq e^{C(|\alpha|+\sqrt{|\lambda|})}\|\bar u\|_{L^2(\mathbb B_{r}(x))}
\label{goback}
\end{equation}
for any $x\in \tilde{\Omega}$.
\label{pro2}
\end{proposition}

\begin{proof}
Let $R=\frac{\bar r_0}{8}$, where $\bar r_0$ is the fixed constant in the three-ball inequality in (\ref{quanti}). Choose $0<\rho<\frac{R}{24}$, which can be chosen to be arbitrarily small. Define a smooth cut-off function $0<\eta<1$ as follows,
\begin{itemize}
\item $\eta(r)=0$ \ \ \mbox{if} \ $r(x)<\rho$ \ \mbox{or} \  $r(x)>2R$, \medskip
\item $\eta(r)=1$ \ \ \mbox{if} \ $\frac{3\rho}{2}<r(x)<R$, \medskip
\item $|\nabla\eta|\leq \frac{C}{\rho}$ \ \ \mbox{if} $\rho<r(x)<\frac{3\rho}{2}$, \medskip
\item $|\nabla^2 \eta|\leq C$ \ \  \mbox{if} \ $R<r(x)<2R$.
\end{itemize}

 We substitute $v=\eta \bar u $ into the Carleman estimates (\ref{carl})  and consider  the elliptic equations (\ref{target}). It follows that
\begin{align*}
 \tau^\frac{3}{2}\| (\log r)^{-1} e^{\tau \psi}\eta \bar u \|+ \tau^{\frac{1}{2}} \rho^{\frac{1}{2}} \| r^{-\frac{1}{2}} e^{\tau \psi}\eta \bar u \|
& \leq C\| r^2 e^{\tau \psi}(\triangle_{\hat{g}} (\eta\bar u)+ \hat{b}(x) \cdot \nabla (\eta\bar u)+ \hat{c}(x) \eta \bar u )\| \nonumber \\
& \leq C\| r^2 e^{\tau \psi}(\triangle \eta \bar u +2\nabla\eta\cdot \nabla v+\hat{b}\cdot \nabla \eta \bar u )\|.
\end{align*}

Thanks to the properties of $\eta$ and the fact that $\tau>1$, we get that
\begin{align*}
 \| (\log r)^{-1} e^{\tau \psi} \bar u \|_{\frac{R}{2}, \frac{2R}{3}}+  \|  e^{\tau \psi} \bar u \|_{\frac{3\rho}{2}, 4\rho}
 &\leq C (\| e^{\tau \psi} \bar u \|_{\rho, \frac{3\rho}{2}}+\| e^{\tau \psi} \bar u\|_{R, 2R} ) \\
&+ C( \rho \| e^{\tau \psi} \nabla \bar u \|_{\rho, \frac{3\rho}{2}}+ R\| e^{\tau \psi} \nabla \bar u \|_{R, 2R}) \\
&+ C(|\alpha|+1) (\rho \| e^{\tau \psi} \bar u \|_{\rho, \frac{3\rho}{2}}+ R\| e^{\tau \psi} \bar u \|_{R, 2R} ).
\end{align*}
Since $R<1$ is a fixed constant and $\rho<1$, we get that
\begin{align*}
\|  e^{\tau \psi} \bar u \|_{\frac{R}{2}, \frac{2R}{3}}+  \|  e^{\tau \psi} \bar u \|_{\frac{3\rho}{2}, 4\rho}
 &\leq C(|\alpha|+1) (\| e^{\tau \psi} \bar u \|_{\rho, \frac{3\rho}{2}}+\| e^{\tau \psi} \bar u \|_{R, 2R} ) \\
&+ C( \delta \| e^{\tau \psi} \nabla \bar u \|_{\rho, \frac{3\rho}{2}}+ R\| e^{\tau \psi} \nabla \bar u \|_{R, 2R}).
\end{align*}
Using the radial and decreasing property of $\psi$ yields that
\begin{align*}
e^{\tau \psi(\frac{2R}{3})}\|  \bar u \|_{\frac{R}{2}, \frac{2R}{3}}+ e^{\tau \psi({4\rho})}  \| \bar u \|_{\frac{3\rho}{2}, 4\rho}
&\leq C (|\alpha|+1) (e^{\tau \psi(\rho) }\|  \bar u \|_{\rho, \frac{3\rho}{2}}+e^{\tau \psi(R) }\|  \bar u \|_{R, 2R} ) \\
&+ C( \rho e^{\tau \psi(\rho)} \| \nabla \bar u \|_{\rho, \frac{3\rho}{2}}+ Re^{\tau\psi(R)}\|  \nabla \bar u \|_{R, 2R}).
\end{align*}
With the help of the Caccioppoli type inequality (\ref{Cacc}) in the appendix, we have
\begin{align}
e^{\tau \psi(\frac{2R}{3})}\|  \bar u \|_{\frac{R}{2}, \frac{2R}{3}}+ e^{\tau \psi({4\rho})}  \|  \bar u \|_{\frac{3\rho}{2}, 4\rho}
&\leq C (|\alpha|+\sqrt{|\lambda|}) (e^{\tau \psi(\rho) }\| \bar u \|_{2\rho}+e^{\tau \psi(R)}\|  \bar u \|_{3R}).
\label{bbb}
\end{align}
Adding the term $ e^{\tau \psi({4\rho})} \|\bar u \|_{\frac{3\rho}{2}}$ to both sides of last inequality and taking $\psi(\rho)>\psi(4\rho)$ into account yields that
\begin{align}
e^{\tau \psi(\frac{2R}{3})}\| \bar u \|_{\frac{R}{2}, \frac{2R}{3}}+ e^{\tau \psi({4\rho})}  \|  \bar u \|_{4\rho}
&\leq C (|\alpha|+\sqrt{|\lambda|}) (e^{\tau \psi(\rho) }\|  \bar u \|_{2\rho}+e^{\tau\psi(R)}\| \bar u \|_{3R}).
\label{droppp}
\end{align}
We choose $\tau$ such that
$$C(|\alpha|+\sqrt{|\lambda|})e^{\tau\psi(R)}\|\bar u \|_{3R}\leq \frac{1}{2}e^{\tau\psi(\frac{2R}{3})} \|\bar u \|_{\frac{R}{2}, \frac{2R}{3}}. $$
To achieve it, we need to have
$$ \tau\geq \frac{1}{\psi(\frac{2R}{3})-\psi(R)}\ln \frac{ 2C(|\alpha|+\sqrt{|\lambda|}) \|\bar u \|_{3R}}{ \|\bar u \|_{\frac{R}{2}, \frac{3R}{2}}}.           $$
Then, we can absorb the second term on the right hand side of (\ref{bbb}) into the left hand side,
\begin{align}
e^{\tau \psi(\frac{2R}{3})}\|  \bar u \|_{\frac{R}{2}, \frac{2R}{3}}+ e^{\tau \psi({4\rho})}  \|  \bar u \|_{4\rho}
&\leq C (|\alpha|+\sqrt{|\lambda|}) e^{\tau \psi(\rho) }\| \bar u \|_{2\rho}.
\label{dropp}
\end{align}
To apply the Carleman estimates (\ref{carl}), we have assumed that $\tau\geq C(|\alpha|+\sqrt{|\lambda|})$.  Therefore, to have such $\tau$, we select
$$\tau=C(|\alpha|+\sqrt{|\lambda|})+ \frac{1}{\psi(\frac{2R}{3})-\psi(R)}\ln \frac{ 2C(|\alpha|+\sqrt{|\lambda|}) \|\bar u \|_{3R}}{ \|\bar u \|_{\frac{R}{2}, \frac{3R}{2}}}. $$
Dropping the first term in (\ref{dropp}), we get that
\begin{align}
\|\bar u\|_{4\rho}&\leq C(|\alpha|+\sqrt{|\lambda|})\exp\{ \big( \frac{1}{\psi(\frac{2R}{3})-\psi(R)}\ln \frac{ 2C(|\alpha|+\sqrt{|\lambda|}) \|\bar u \|_{3R}}{ \|\bar u \|_{\frac{R}{2}, \frac{3R}{2}}} \big)\big(\psi(\rho)-\psi(4\rho)\big) \nonumber \\ &  +C(|\alpha|+\sqrt{|\lambda|}) \}\|\bar u \|_{2\rho} \nonumber \\
&\leq e^{C (|\alpha|+\sqrt{|\lambda|})}   (\frac{\|\bar u \|_{3R}}{ \|\bar u \|_{\frac{R}{2}, \frac{3R}{2}}})^C \|\bar u \|_{2\rho},
\label{mada}
\end{align}
where we have used the condition that
$$ \beta_1^{-1}<\psi(\frac{2R}{3})-\psi(R)<\beta_1,  $$
$$  \beta_2^{-1}< \psi(\rho)-\psi(4\rho)<\beta_2 $$
for some positive constants $\beta_1$ and $\beta_2$ independent on $R$ or $\rho$.

Let $\hat{r}_0=\frac{R}{2}$ be fixed in (\ref{annu1}).
With aid of (\ref{annu1}), we derive that
$$ \frac{\|\bar u \|_{{L^2(\mathbb B_{3 R}})}}{ \|\bar u\|_{{{L^2(A_{\frac{R}{2}, \ \frac{3R}{2}})}}}}\leq e^{C(|\alpha|+\sqrt{|\lambda|})}. $$
Therefore, it follows from (\ref{mada}) that
$$ \|\bar u \|_{L^2(\mathbb B_{4\rho})}\leq  e^{C(|\alpha|+\sqrt{|\lambda|})} \|\bar u \|_{L^2(\mathbb B_{2\rho})}. $$
Choosing $\rho=\frac{r}{2}$, we get the doubling inequality
\begin{equation}
\|\bar u \|_{L^2(\mathbb B_{2r})}\leq  e^{C(|\alpha|+\sqrt{|\lambda|})} \|\bar u \|_{L^2(\mathbb B_{r})}
\label{rdoub1}
\end{equation}
for $r\leq \frac{R}{12}$. If $r\geq \frac{ R}{12}$, from $(\ref{useful})$,
\begin{align}
\|\bar u \|_{L^2 (\mathbb B_r)}&\geq \|\bar u \|_{L^2( \mathbb B_\frac{ R}{12})} \nonumber \\
&\geq  e^{-C_{ R}(|\alpha|+\sqrt{|\lambda|})}\|\bar u \|_{L^2(\Omega)}  \nonumber  \\
&\geq e^{-C_{ R}(|\alpha|+\sqrt{|\lambda|})}\|\bar u \|_{L^2( \mathbb B_{2r})}.
\label{rdoub2}
\end{align}
Together with (\ref{rdoub1}) and (\ref{rdoub2}), we obtain the doubling estimates
\begin{equation}
\|\bar u \|_{L^2(\mathbb B_{2r})}\leq  e^{C(|\alpha|+\sqrt{|\lambda|})} \|\bar u \|_{L^2(\mathbb B_{r})}
\label{dddou}
\end{equation}
 for any $r>0$, where $C$ only depends on the double manifold $\tilde{\Omega}$. By the translation invariant of the arguments, the proof of (\ref{goback}) is derived.

 \end{proof}
For the latter sections on the study of nodal sets and the doubling inequality on the boundary, we want to show the doubling inequality in the half ball $\mathbb B^+_r(0)$.
 Since we did an even extension across the boundary $\{x_n=0\}$ and the metric $\hat{g}_{ij}$ is  symmetric with $\{x_n=0\}$,  the estimates (\ref{goback}) also holds in the half balls. Thus,
 there exist  positive constants $C$, $r_0$ depending only on ${\Omega}$ such that the doubling inequality holds
\begin{equation}
\|\bar u \|_{L^2(\mathbb B^+_{2r})}\leq e^{C(|\alpha|+\sqrt{|\lambda|})}\|\bar u \|_{L^2(\mathbb B^+_{r})}
\label{halfdouble}
\end{equation}
for $0<r<{r_0}$.

Thanks to Proposition \ref{pro2}, we are ready to show the upper bound of nodal sets of Robin eigenfunctions in smooth domains $\Omega$.

\begin{proof}[Proof of Theorem \ref{th1}]
We estimate the nodal sets of $\bar u$ for the converted  second order elliptic equations
(\ref{back}). The doubling inequality (\ref{goback}) in Proposition \ref{pro2} implies that the double index defined in \cite{Lo} is less than $C(|\alpha|+\sqrt{|\lambda|})$. Following the strategy in \cite{Lo} which holds for second order elliptic equations with Lipschitz leading coefficients, we can derive the upper bound of the measure of nodal sets on $\tilde{\Omega}$,
\begin{equation} H^{n-1}(\{x\in\tilde{\Omega} |\bar u(x)=0\})\leq C(|\alpha|+\sqrt{|\lambda|})^{\beta}
 \end{equation}
for some $\beta>1$. By the explicit form of $\bar u$ and the even extension, the upper bound of nodal sets in the $\Omega$ easily follows
\begin{equation} H^{n-1}(\{x\in\Omega | u(x)=0\})\leq C(|\alpha|+\sqrt{|\lambda|})^{\beta}.
 \end{equation}
This completes the proof of Theorem \ref{th1}.
\end{proof}

\begin{remark}
In $n=2$ dimensions, an explicit upper bound can be obtained. Using the recent development on combinatorial arguments for doubling index in \cite{LM} and the nodal length estimates in small scales in \cite{DF} or \cite{Zh1}, we can show that
\begin{equation} H^{1}(\{x\in\Omega | u(x)=0\})\leq C(|\alpha|+\sqrt{|\lambda|})^{\frac{3}{4}-\epsilon}
 \end{equation}
for some small constant $\epsilon>0$.
\end{remark}

\section{Interior nodal sets on real analytic domains}

In this section, we prove the  upper bound for interior nodal sets of
Robin eigenfunctions.  Assume that $\Omega$ is a real analytic domain. If $\lambda=0$, the Robin eigenvalues problem can be reduced to the Steklov eigenvalue problem as discussed in Remark \ref{rre2}. The measure of interior nodal sets for analytic domains has been obtained in \cite{Zh2}. Hence, we assume $\lambda\not=0$ in the section. We first show the upper bound of
nodal sets in the neighborhood close to boundary, then show the upper
bound of nodal sets away from the boundary $\partial\Omega$. To deal with the nodal sets close to the boundary, we do an analytical extension across the boundary using lifting arguments and the Cauchy-Kovalevsky theorem.

To get rid of $\alpha$ on the boundary, we introduce the following lifting argument.
Let $$\hat{v}(x, t)=e^{\alpha t} u(x).$$
Then $\hat{v}(x, t)$ satisfies the equation
\begin{equation}
\left \{ \begin{array}{rll}
\triangle \hat{v} +\partial_t^2 \hat{v}-\alpha^2 \hat{v}+\lambda \hat{v}=0 \quad &\mbox{in} \ {\Omega}\times (-\infty, \ -\infty),  \medskip\\
\frac{\partial \hat{v}}{\partial \nu}+\frac{\partial \hat{v}}{\partial t}=0 \quad &\mbox{on} \ {\partial\Omega}\times (-\infty, \ -\infty).
\end{array}
\right.
\end{equation}
To remove the $\alpha$ and $\lambda$ in the equation, we
choose $${v}(x, t, s)=e^{(|\alpha|i+\sqrt{\lambda})s} \hat{v}(x, t).$$
If $\lambda<0$, then $\sqrt{\lambda}$ is considered as the imaginary number.
Then we have
\begin{equation}
\left \{ \begin{array}{rll}
\triangle {v} +\partial_t^2 {v}+ \partial_s^2 {v}=0 \quad &\mbox{in} \ {\Omega}\times (-\infty, \ -\infty)\times (-\infty, \ -\infty),  \medskip\\
\frac{\partial {v}}{\partial \nu}+\frac{\partial {v}}{\partial t}=0 \quad &\mbox{on} \ {\partial\Omega}\times (-\infty, \ -\infty)\times (-\infty, \ -\infty).
\end{array}
\right.
\label{killz}
\end{equation}
 Note that (\ref{killz}) is a uniform elliptic equation with oblique boundary conditions. We
introduce the cube with unequal radius as
\begin{align*}
\Omega_{R, \delta}=\{ (x, t, s)\in \mathbb R^{n+2}| |x_i|<R \ \mbox{when} \ i<n, |x_n|<\delta R, \ |t|<R, \ |s|<R\}
\end{align*}
and half-cube
\begin{align*}
\Omega^+_{R, \delta}=\{ (x, t, s)\in \mathbb R^{n+2} | |x_i|<R \ \mbox{when} \ i<n, 0\leq x_n<\delta R, \ |t|<R, \ |s|<R\}.
\end{align*}
 By rescaling, we may consider the function $v(x, t, s )$  locally in the cube with the flatten boundary using the Fermi coordinates in Section 2.  Thus, $v(x, t, s )$ satisfies
\begin{equation}
\left \{ \begin{array}{lll}
\triangle {v} +\partial_t^2 {v}+ \partial_s^2 {v}=0 \quad &\mbox{in} \ \Omega^+_{2, 1},  \medskip\\
\frac{\partial {v}}{\partial x_n}+\frac{\partial {v}}{\partial t}=0 \quad &\mbox{on} \  \Omega^+_{2, 1}\cap\{x_n=0\}.
\end{array}
\right.
\label{backgo}
\end{equation}
By elliptic estimates with oblique boundary conditions, it follows that
\begin{align}
\| v\|_{C^4(\Omega^+_{\frac{3}{2}, 1})}\leq C_1(\Omega) \| v\|_{L^\infty(\Omega^+_{2, 1})},
\label{niren}
\end{align}
where $C_1(\Omega)$ depends only on $\Omega$.
We apply the Cauchy-Kovalevsky theorem in \cite{H} to extend $v$ across the boundary $\partial\Omega\times [-1, \ 1]\times [-1, \ 1]$. We can transform the boundary conditions in (\ref{backgo}) to have the homogeneous boundary conditions. With aid of the Cauchy-Kovalevsky theorem (Theorem 9.4.5) in \cite{H}, we can extend $v(x, t, s)$ to the region $\Omega_{1, \delta}$, where $\delta$ depends only on $\partial\Omega$. Moreover, we have
\begin{align}
\|v\|_{L^\infty(\Omega_{{1}, \delta})}
\leq C_2(\Omega)\|v\|_{L^\infty (\Omega^+_{2, 1})}.
\label{agree}
\end{align}

By the uniqueness of the analytic continuation, it also holds that
\begin{align}
-\triangle u=\lambda u \quad &\mbox{in} \ \widehat{\Omega}_1,
\end{align}
where $\widehat{\Omega}_1= \{ x\in \mathbb R^n| \dist\{x, \Omega\}\leq \delta\}$.

Recall the definition of ${v}(x, t, s)=e^{(|\alpha|i+\sqrt{\lambda})s} e^{\alpha t} u(x)$, from (\ref{agree}), it follows that
\begin{align}
\|u\|_{L^\infty(\mathbb B_{\delta})}\leq e^{C(|\alpha|+\sqrt{|\lambda|})} \|u\|_{L^\infty(\mathbb B^+_2)}.
\end{align}
From the relations of $\bar u$ and $u$, the doubling inequality as (\ref{halfdouble}) also holds for $u$ as
\begin{equation}
\|u \|_{L^2(\mathbb B^+_{2r})}\leq e^{C(|\alpha|+\sqrt{|\lambda|})}\| u \|_{L^2(\mathbb B^+_{r})}
\label{halfdouble2}
\end{equation}
for $0<r<{r_0}$.
Applying the doubling inequality (\ref{halfdouble2}) in $L^\infty$ norm finitely many times in the half balls, we have
\begin{align}
\|u\|_{L^\infty(\mathbb B_{\delta})}&\leq e^{C(|\alpha|+\sqrt{|\lambda|})} \|u\|_{L^\infty(\mathbb B^+_\frac{\delta}{2})} \nonumber \\
&\leq e^{C(|\alpha|+\sqrt{|\lambda|})} \|u\|_{L^\infty(\mathbb B_\frac{\delta}{2})}.
\end{align}
By rescaling arguments, we show that
\begin{align}
\|u\|_{L^\infty(\mathbb B_{2r})}\leq e^{C(|\alpha|+\sqrt{|\lambda|})} \|u\|_{L^\infty(\mathbb B_r)}.
\label{correct}
\end{align}
for any $r\leq \frac{\delta}{2}$ and $C$ depending only on $\Omega$.

We want to extend ${u}(x)$ locally as a holomorphic function in
$\mathbb{C}^n$. Applying elliptic estimates in a ball $\mathbb B_{
\frac{r}{({|\alpha|}+\sqrt{|\lambda|})}}(p)\subset \widehat{\Omega}_1$, we have
\begin{equation}
|\frac{ D^{\bar{\alpha}} {u}(p)}{\bar{\alpha} !}|\leq
C^{|\bar \alpha|}_1({|\alpha|}+\sqrt{|\lambda|})^{|\bar\alpha|} r^{-|\bar{\alpha}| }\|{u}\|_{L^\infty},
\end{equation}
where $\bar{\alpha}$ is a multi-index and $C_1>1$ depends on $\Omega$.
We may consider the point $p$ as the origin.
If summing a geometric series, we can extend ${u}(x)$ to be a
holomorphic function ${u}(z)$ with $z\in\mathbb{C}^n$. Moreover,
we derive that
\begin{equation}
\sup_{|z|\leq \frac{(|\alpha|+\sqrt{|\lambda|})^{-1}r}{2C_1}}|{u}(z)|\leq C_2 \sup_{|x|\leq
(|\alpha|+\sqrt{|\lambda|})^{-1}r}|{u}(x)|
\end{equation}
with $C_2>1$. Iterating $|\alpha|+\sqrt{|\lambda|}$ times to the translation of ${u}$ in the complex direction gives that
$$\sup_{|z|\leq r}|{u}(z)|\leq e^{C_4(|\alpha|+\sqrt{|\lambda|})} \sup_{|x|\leq
C_3 r}|{u}(x)|,
  $$
where $C_3, C_4>1$ depends on $\Omega$.
  Taking the doubling inequality (\ref{correct})
 in account, from the rescaling argument,  gives that
\begin{equation}
\sup_{|z|\leq 2r}|{u}(z)|\leq e^{C_5(|\alpha|+\sqrt{|\lambda|})} \sup_{|x|\leq
r}|{u}(x)| \label{dara}
\end{equation}
for $0<r<r_0$ with $r_0$ depending on $\Omega$ and $C_5$ depends on $\Omega$.

We need a lemma concerning the growth of a complex analytic function
with the number of zeros, e.g. Lemma 2.3.2 in \cite{HL}.
\begin{lemma}
Suppose $f: \mathcal{B}_1(0)\subset \mathbb{C}\to \mathbb{C}$ is an
analytic function satisfying
$$ f(0)=1\quad \mbox{and} \quad \sup_{ \mathcal{B}_1(0)}|f|\leq 2^N$$
for some positive constant $N$. Then for any $r\in (0, 1)$, there
holds
$$\sharp\{z\in\mathcal{B}_r(0): f(z)=0\}\leq cN        $$
where $c$ depends on $r$. Especially, for $r=\frac{1}{2}$, there
holds
$$\sharp\{z\in \mathcal{B}_{1/2}(0): f(z)=0\}\leq N.        $$
\label{wwhy}
\end{lemma}
With Lemma \ref{wwhy} and the doubling inequalities, we are ready to give the proof of Theorem 2. The ideas originate from the pioneering work in  \cite{DF} and \cite{Lin}.
\begin{proof}[Proof of Theorem 2] In the first step, we prove the nodal sets in a
neighborhood $\Omega_{\frac{\delta}{8}}$. Since $\partial\Omega$ is compact, we can choose a sequence of finite number of balls centered on $\partial\Omega$ such that those balls cover $\Omega_{\frac{\delta}{8}}$. By rescaling and
translation, we can argue on scales of order one and choose balls centered at origin. Let $p\in \mathbb
B_{1/4}$ be the point where the maximum of $|{u}|$ in $\mathbb
B_{1/4}$ is attained. After a rescaling, we may assume that ${u}(p)=1$. By the doubling inequality (\ref{dara}), we have
\begin{align}
\sup_{|z|\leq 1}|{u}(z)|&\leq e^{C(|\alpha|+\sqrt{|\lambda|})} \sup_{|x|\leq
\frac{1}{4}}|{u}(x)|\nonumber \\ &\leq  e^{C(|\alpha|+\sqrt{|\lambda|})}.
\end{align}
Applying (\ref{dara}) to the translation of ${u}$, we obtain that
\begin{align}
\sup_{|z-p|\leq 1}|{u}(z)|&\leq e^{C(|\alpha|+\sqrt{|\lambda|})} \sup_{|x-p|\leq
\frac{1}{4}}|{u}(x)|  \nonumber \\
&\leq  e^{C(|\alpha|+\sqrt{|\lambda|})} \sup_{|x|\leq
\frac{1}{2}}|{u}(x)| \nonumber \\ &\leq e^{C(|\alpha|+\sqrt{|\lambda|})}.
\label{dara8}
\end{align}
 For each direction $\omega \in S^{n-1}$, set
${u}_{\omega}(z)={u}(p+z\omega)$ in $z\in \mathcal{B}_1(0)\subset\mathbb{C}$. Denote $N(\omega)=\sharp\{z\in\mathcal{B}_{1/2}(0)\subset
\mathbb{C}| e_\omega(z)=0\}$. By the doubling property (\ref{dara8}) and
Lemma \ref{wwhy}, we can show that
\begin{eqnarray}
\sharp\{ x \in \mathbb B_{1/2}(p) &|& x-p \ \mbox{is parallel to} \
\omega \ \mbox{and} \ {u}(x)=0\} \nonumber\\&\leq&
\sharp\{z\in\mathcal{B}_{1/2}(0)\subset
\mathbb{C}| {u}_\omega(z)=0\} \nonumber\\
&=&N(\omega)\nonumber\\
&\leq& C(|\alpha|+\sqrt{|\lambda|}).
\end{eqnarray}
With aid of integral geometry estimates, it yields that
\begin{eqnarray}
H^{n-1}\{ x \in \mathbb B_{1/2}(p)| {u}(x)=0\} &\leq&
c(n)\int_{S^{n-1}} N(\omega)\,d \omega \nonumber\\
&\leq &\int_{S^{n-1}}C(|\alpha|+\sqrt{|\lambda|})\, d\omega \nonumber\\
&=&C(|\alpha|+\sqrt{\lambda}).
\end{eqnarray}
Since $\mathbb B_{1/4}(0)\subset \mathbb B_{1/2}(p)$, we obtain the upper bound estimates
\begin{align}
H^{n-1}\{ x \in \mathbb B_{1/4}(0)| {u}(x)=0\} \leq C(|\alpha|+\sqrt{\lambda}).
\end{align}
By covering the domain
$\Omega_{\frac{\delta}{8}}\subset\widehat{\Omega}_1$ by finite number
of balls centered at $\partial\Omega$. Then we obtain that \begin{equation}H^{n-1}\{x\in
\Omega_{\frac{\delta}{8}}|{u}(x)=0\}\leq C(|\alpha|+\sqrt{|\lambda|}).
\label{last12}
\end{equation}

In the second step, we deal with the measure of nodal sets in
$\Omega\backslash \Omega_{\frac{\delta}{8}}$. Recall that we have
obtained the doubling inequality in the interior of the domain in Proposition \ref{pro2},
i.e.
$$
\|\bar u\|_{L^\infty(\mathbb B_{2r}(p))}\leq
e^{C(|\alpha|+\sqrt{|\lambda|})}\|\bar u\|_{L^\infty(\mathbb B_{r}(p))}.
$$
Since $\bar u(x)= u(x)\exp\{-\alpha \varrho(x)\}$ and $-C_0
<\varrho(x)\leq C_0$ for some constant $C_0$ depending on
$\Omega$ in (\ref{441}), it follows that
\begin{equation}
\|u \|_{L^\infty(\mathbb B_{2r}(p))}\leq
e^{C(|\alpha|+\sqrt{\lambda})}\|u\|_{L^\infty(\mathbb B_{r}(p))}
\label{corre}
\end{equation}
holds for $p\in \Omega\backslash \Omega_{\frac{\delta}{8}}$
and $0<r\leq r_0\leq \frac{\delta}{8}$. We can similarly extend
$u(x)$ locally as a holomorphic function in $\mathbb {C}^n$.
 Applying elliptic estimates in a
small ball $\mathbb B_{\frac{r}{|\alpha|+\sqrt{|\lambda|}}}(p)\subset \Omega\backslash
\Omega_{\frac{\delta}{8}}$ yields that
\begin{equation}
|\frac{ D^{\bar \alpha} u(p)}{\bar{\alpha} !}|\leq
C_6^{|\bar \alpha|}(|\alpha|+\sqrt{|\lambda|})^{\frac{|\bar\alpha|}{2}}r^{-|\bar\alpha|}\|u\|_{L^\infty},
\end{equation}
with $C_6>1$ depending on $\Omega$.
Let us consider the point $p$ as the origin. Summing up a geometric
series implies that we can extend $u(x)$ to be a holomorphic function
$u(z)$ with $z\in\mathbb{C}^n$ to have
\begin{equation}
\sup_{|z|\leq \frac{r}{2C_6 (|\alpha|+\sqrt{|\lambda|}) }}|u(z)|\leq C_7 \sup_{|x|\leq
\frac{r}{|\alpha|+\sqrt{|\lambda|}}}|u(x)|
\end{equation}
with $C_7>1$.
Furthermore, it holds that
\begin{equation}
\sup_{|z|\leq r} |u(z)|\leq e^{C_9(|\alpha|+\sqrt{|\lambda|})} \sup_{|x|\leq
{C_8r}}|u(x)|
\end{equation}
by iterating $|\alpha|+\sqrt{|\lambda|}$ times to the translation of $u(x)$ in the complex direction, where $C_8, C_9>1$ depends on $\Omega$.
By iterating the doubling inequality (\ref{corre}) finite number steps and rescaling arguments, we arrive at
\begin{equation}
\sup_{|z|\leq 2r}|u(z)|\leq e^{C(|\alpha|+\sqrt{|\lambda|})} \sup_{|x|\leq
r}|u(x)| \label{dara1}
\end{equation}
holds for $0<r<r_0$ with $r_0$ depending on $\Omega$ and $C$ independent of $\lambda$ and $r$.

Following the
same procedure in the neighborhood of the
boundary, we make use of Lemma \ref{wwhy} and the inequality
(\ref{dara1}) to obtain  the upper bounds of the nodal sets in the interior domain.  By rescaling and translation, we also argue on scales
of order one. Let $p\in \mathbb B_{1/4}$ be the point where the
maximum of $|u|$ in $\mathbb B_{1/4}$ is achieved. For each
direction $\omega \in S^{n-1}$, let $u_{\omega}(z)=
u(p+z\omega)$ in $z\in \mathcal{B}_1(0)\subset\mathbb{C}$.
With aid of the doubling property (\ref{dara1}) and the Lemma \ref{wwhy}, we have that
\begin{eqnarray}
\sharp\{ x \in \mathbb B_{1/2}(p) &|& x-p \ \mbox{is parallel to} \
\omega \ \mbox{and} \ u(x)=0\} \nonumber\\&\leq&
\sharp\{z\in\mathcal{B}_{1/2}(0)\subset
\mathbb{C}| u_\omega(z)=0\} \nonumber\\
&=&N(\omega)\nonumber\\ &\leq& C(|\alpha|+\sqrt{|\lambda|}).
\end{eqnarray}
Thanks to the integral geometry estimates, we get
\begin{eqnarray}
H^{n-1}\{ x \in \mathbb B_{1/2}(p)| u(x)=0\} &\leq&
c(n)\int_{S^{n-1}} N(\omega)\,d \omega \nonumber\\
&\leq &\int_{S^{n-1}}C(|\alpha|+\sqrt{\lambda})\, d\omega\nonumber\\ &=&C(|\alpha|+\sqrt{|\lambda|}).
\end{eqnarray}
 Covering the domain $\Omega\backslash\Omega_{\frac{\delta}{8}}$ using finite number of balls gives that
\begin{equation} H^{n-1}\{x\in
\Omega\backslash\Omega_{\frac{\delta}{8}}|u(x)=0\}\leq
C(|\alpha|+\sqrt{|\lambda|}). \label{last2}
\end{equation}
Combining the results in (\ref{last12}) and (\ref{last2}), we arrive at the
conclusion in Theorem \ref{th4}.

\end{proof}

\section{Boundary doubling inequality}

In this section, we prove new quantitative propagation smallness results for the second order elliptic equations (\ref{half})
in the half ball. By rescaling, we may consider the equations in $\mathbb B^+_{1/2}$. To present the results in a general setting, we may consider the second order uniformly elliptic equations
\begin{align}
- a_{ij} D_{ij}u +b_i(y) D_i u +c(y)u=0  \quad{in} \ \mathbb B^+_{1/2},
\label{general}
\end{align}
where $a_{ij}$ is $C^1$, and $ b(y)$ and $ c(y)$ satisfy
\begin{equation}
\left\{ \begin{array}{lll}
\| {b}\|_{W^{1, \infty} (\mathbb B^+_{1/2})}\leq C(|\alpha|+1), \medskip \\
\| {c}\|_{W^{1, \infty}(\mathbb B^+_{1/2})}\leq C(\alpha^2 +|\lambda|).
\end{array}
\right.
\label{aaa}
\end{equation}

We are able to show the following quantitative two half-ball and one lower dimensional ball type result.
\begin{lemma}
Let $u\in C^\infty_0(\mathbb B^+_{1/2}  )$ be a solution of (\ref{general}). Denote
$${\bf B}_{1/3} =\{ (y', \ 0)\in \mathbb R^n| y'\in \mathbb R^{n-1}, \ |y'|<\frac{1}{3}\}.$$
Assume that
\begin{equation}
\|u \|_{H^1({\bf B}_{1/3})}+\|\frac{\partial u}{\partial \nu}\|_{L^2({\bf B}_{1/3})}\leq \epsilon <<1
\end{equation}
and $\|u \|_{L^2(\mathbb B^+_{1/2})}\leq 1$. There exist positive constants $C$ and $\beta$ such that
\begin{equation}
\|u\|_{L^2(\frac{1}{256}\mathbb B^+_1)}\leq e^{C(|\alpha|+\sqrt{|\lambda|})} \epsilon^\beta.
\label{too}
\end{equation}
More precisely, we can show that there exists $0<\kappa<1$ such that
\begin{equation}
\|u\|_{L^2(\frac{1}{256}\mathbb B^+_1)}\leq e^{C(|\alpha|+\sqrt{|\lambda|})} \|u\|_{L^2(\mathbb B^+_{1/2})}^\kappa \big( \|u \|_{H^1({\bf B}_{1/3})}+\|\frac{\partial u}{\partial \nu}\|_{L^2({\bf B}_{1/3})}  \big)^{1-\kappa}.
\label{tool}
\end{equation}
\label{halfspace}
\end{lemma}

Such estimates without considering the quantitative behavior of $\alpha$ and $\lambda$ has been established in \cite{Lin}. To show the quantitative three-ball inequality in the lemma, we develop some novel quantitative global Carleman estimates  involving the boundary. The weight function in Carleman estimates (\ref{useew}) is somewhat inspired by \cite{LR} and \cite{JL}.  Such results play an important role not only in characterizing the doubling index in a cube in \cite{Lo}, but also in inverse problems, see \cite{ARRV}.

The quantitative global Carleman estimates with boundary  is stated in Proposition \ref{Pro4}.
We choose a weight function $$\psi(y)= e^{sh(y)},$$ where $$h(y)=-\frac{|y'|^2}{4}+\frac{y^2_n}{2}-y_n$$ and $s$ is a large parameter that will be determined later.
\begin{proposition}
Let $s$ be a fixed large constant. There exist positive constants $C_s$ and $C_0$ depending on $s$ such that for any  $v\in C^\infty(\mathbb B_{1/2}^+)$, and
$$\tau>C_s(|\alpha|+\sqrt{|\lambda|}),$$
one has
\begin{align}
\| e^{\tau \psi}(- a_{ij} D_{ij}v +b_i D_i v +c v)\|_{L^2(\mathbb B_{1/2}^+)} &+\tau^\frac{3}{2} s^2\| \psi^{\frac{3}{2}} e^{\tau \psi} v \|_{L^2(\partial \mathbb B_{1/2}^+)} + \tau^\frac{1}{2} s \| \psi^{\frac{1}{2}} e^{\tau \psi}\nabla v \|_{L^2(\partial \mathbb B_{1/2}^+)} \nonumber \\
&\geq C_0 \tau^\frac{3}{2}s^2 \|\psi^{\frac{3}{2}} e^{\tau \psi} v \|_{L^2( \mathbb B_{1/2}^+)}+ C_0 \tau^\frac{1}{2} s \| \psi^{\frac{1}{2}}  e^{\tau \psi} \nabla v \|_{L^2( \mathbb B_{1/2}^+)}.
 \label{useew}
\end{align}
\label{Pro4}
\end{proposition}

Since the proof of Proposition \ref{Pro4} is lengthy, we postpone the proof in Section 6. Thanks to the Carleman estimates (\ref{useew}), we first show the proof of Lemma \ref{halfspace}.
\begin{proof}[Proof of Lemma \ref{halfspace}]
 Notice that the constant $s$ is fixed independent of $\alpha$ and $\lambda$. We also know $\psi$ is bounded below and above by some constant $C$. We  obtain that
\begin{align}
\| e^{\tau \psi} (- a_{ij} D_{ij}v +b_i D_i v +c v) \|_{L^2(\mathbb B_{1/2}^+)} &+\tau^\frac{3}{2}\| e^{\tau \psi} v \|_{L^2(\partial \mathbb B_{1/2}^+)} + \tau^\frac{1}{2}\| e^{\tau \psi}\nabla v \|_{L^2(\partial \mathbb B_{1/2}^+)} \nonumber \\
&\geq C \tau^\frac{3}{2}\| e^{\tau \psi} v \|_{L^2( \mathbb B_{1/2}^+)}+ C \tau^\frac{1}{2}\| e^{\tau \psi} \nabla v \|_{L^2( \mathbb B_{1/2}^+)}.
 \label{usee}
\end{align}

 The following Caccioppolli inequality holds for the solutions of (\ref{general}) in $\mathbb B^+_{1/2}$,
\begin{align}
 \|\nabla u\|_{L^2(\mathbb B_r^+)} \leq \frac{C(|\alpha|+\sqrt{|\lambda|})}{r}\big ( \| u\|_{L^2(\mathbb B_{2r}^+)}+ \|\nabla u\|_{L^2({\bf B}_{2r})}+ \|{ u}\|_{L^2( {\bf B}_{2r})}\big).
 \label{cacciop}
\end{align}

We select a smooth cut-off function $\eta$ such that $\eta(x)=1$ in $\mathbb B_{1/8}^+$ and $\eta(x)=0$ outside $\mathbb B_{1/4}^+$. Since $u\in C^\infty_0 (\mathbb B_{1/2}^+ )$, substituting $v$ by $\eta u$ in the Carleman estimates (\ref{usee}) and then using the equation (\ref{general}) yields that
\begin{align}
\| e^{\tau \psi}(-a_{ij} D_{ij} \eta u- 2 a_{ij}D_i \eta D_j u+ b_i D_i \eta u  )\|_{L^2(\mathbb B_{1/2}^+)}&+\tau^{\frac{3}{2}} \| e^{\tau \psi}  \eta u \|_{L^2({\bf B}_{1/4})}\nonumber\\  &+
\tau^{\frac{1}{2}} \| e^{\tau \psi} \eta \nabla u  \|_{L^2({\bf B}_{1/4})}\nonumber\\
&\geq C\tau^{\frac{3}{2}} \| e^{\tau \psi} \eta u \|_{L^2(\mathbb B_{1/2}^+)}.
\label{meimei}
\end{align}

We want to find the maximum of $\psi$ in the first term on the left hand side of (\ref{meimei}). Since $h(y)$ is negative in $\mathbb B^+_{1/2}$, then
$$ \max_{\{\frac{1}{8}\leq |y|\leq \frac{1}{4} \}\cap \{y_n\geq 0\}}h(y) =\max_{\{\frac{1}{8}\leq |y| \leq \frac{1}{4} \}} -\frac{|y'|^2}{4}=-\frac{1}{256}.$$
We also need to find a lower bound of $\psi$ for the term on the right hand side of (\ref{meimei}) such that $$-\min_{|y|<a} h(y)-\frac{1}{256}<0$$
for some $0<a<\frac{1}{2}.$
Since $h(y)$ decreases with respect to $y'$ and $y_n$, then the minimum of $h(y)$ is $\hat{h}(a)$ for $|y|<a$, where
$$\hat{h}(a)=-\frac{a^2}{4}+\frac{a^2}{2}-a=\frac{a^2}{4}-a.  $$
Solving the inequality $-\hat{h}(a)<\frac{1}{256}$, we have one solution $a=\frac{1}{256}$.
Set $$\psi_0=e^{-\frac{s}{256}}- e^{s \hat{h}(\frac{1}{256})},$$
then $\psi_0<0. $
Define $$\psi_1=1- e^{s \hat{h}(\frac{1}{256})}.$$
Since $\hat{h}(\frac{1}{256})<0$, then $\psi_1>0$.

Applying the Caccioppolli inequality (\ref{cacciop}), we arrive at
\begin{align}
\exp \{\tau e^{\frac{-s}{256}}\}\|u\|_{L^2(\mathbb B_{1/2}^+)}&+ e^\tau \|u\|_{L^2({\bf B}_{1/3})}+e^\tau \|\nabla u\|_{L^2({\bf B}_{1/3})}   \nonumber\\
&\geq  C\tau \exp\{\tau e^{s\hat{h}(\frac{1}{256})}\}\|u\|_{L^2(\mathbb B_{{1}/{256}}^+)}.
\end{align}
Let
$$   B_1=\|u\|_{L^2(\mathbb B_{1/2}^+)}, $$
$$  B_2=   \|u\|_{L^2({\bf B}_{1/3})}+ \|\nabla u\|_{L^2({\bf B}_{1/3})} , $$
$$ B_3=\|u\|_{L^2(\mathbb B_{{1}/{256}}^+)}. $$
Multiplying both sides of the last inequality with $\exp\{-\tau e^{s\hat{h}(\frac{1}{256})}\}$ leads to
\begin{equation}
e^{\tau  \psi_0  }B_1+ e^{\tau  \psi_1  }B_2\geq C B_3.
\label{BBB}
\end{equation}
We want to incorporate the first term on the left hand side of (\ref{BBB}) into the right hand side. Let
$$ e^{\tau  \psi_0  }B_1\leq \frac{1}{2} C B_3.  $$
Thus, we need to have
$$\tau\geq \frac{1}{\psi_0 } \ln \frac{CB_3}{2B_1}.  $$
Therefore, for such $\tau$,
\begin{equation}
 e^{\tau  \psi_1  }B_2\geq C B_3.
 \label{mistake}
\end{equation}
Recall that the assumption
$$ \tau\geq C(|\alpha|+\sqrt{|\lambda|})  $$
in Proposition \ref{Pro4}. We assume that
\begin{equation}
\tau= C(|\alpha|+\sqrt{|\lambda|})+\frac{1}{\psi_0 } \ln \frac{CB_3}{2B_1}.
\end{equation}
Note that $\psi_0$ and $\psi_1$ are constants. Substituting such $\tau$ in (\ref{mistake}) yields that
\begin{equation}
e^{C(|\alpha|+\sqrt{|\lambda|})} B_1^{\frac{\psi_1}{\psi_1-\psi_0}} B_2^{\frac{-\psi_0}{\psi_1-\psi_0}}\geq CB_3.
\end{equation}

Let $\kappa =\frac{\psi_1}{\psi_1-\psi_0}$. Then the following three-ball type inequality follows as
\begin{equation}
\|u\|_{L^2(\frac{1}{256}\mathbb B^+_1)}\leq e^{C(|\alpha|+\sqrt{|\lambda|})} \|u\|_{L^2(\mathbb B^+_{1/2})}^\kappa \big( \|u\|_{L^2({\bf B}_{1/3})}+ \|\nabla u\|_{L^2({\bf B}_{1/3})}\big)^{1-\kappa}.
\label{nice}
\end{equation}
Since $u\in C^2(\mathbb B^+_{1/2})$ and $\nabla u=\nabla' u+\frac{\partial u}{\partial \nu}$ on the boundary $\mathbb B^+_{1/2}\cap \{y_n=0\}$, the inequality (\ref{nice}) implies the desired estimates (\ref{tool}).
The estimate (\ref{too}) is a consequence of (\ref{tool}). Therefore, the lemma is finished.

\end{proof}
 We are in the position to prove Theorem \ref{th3} with inspirations from \cite{Lin}. Similar doubling estimates for fractional Laplacian equations on product manifolds were shown in \cite{R}.

\begin{proof}[Proof of Theorem \ref{th3}]
We consider the solution $\bar u$ in the equations (\ref{half}) with conditions (\ref{halfcon}). We argue on scale of order one.
We may normalize $\bar u$ as
\begin{equation} \|\bar u \|_{L^2(\mathbb B^+_{1/2})}=1.
\label{normalize} \end{equation} We claim that there exists a positive constant $C>0$ such that the following lower bound holds on the boundary
\begin{align}
\|\bar u \|_{H^1({\bf B}_{1/6})} \geq e^{ -C(|\alpha|+\sqrt{|\lambda|})}.
\label{haoma}
\end{align}
We will need to use the quantitative three-ball inequality (\ref{tool}) on the half balls. Note that $\frac{\partial \bar u }{\partial \nu}=0$ on the boundary $\{x_n=0\}$. We may normalize the inequality (\ref{tool}) as
\begin{equation}
\|\bar u \|_{L^2(\mathbb B^+_{1/512})}\leq e^{ C(|\alpha|+\sqrt{|\lambda|})} \|\bar u \|_{L^2(\mathbb B^+_{1/4})}^\kappa  \|\bar u \|^{1-\kappa}_{H^1({\bf B}_{1/6})}  .
\label{tooll}
\end{equation}
We prove the claim by contradiction.
If the claim is not true, from (\ref{tooll}), for any constant $\hat{ C}>0$, we have
\begin{equation}
\|\bar u \|_{L^2(\mathbb B^+_{1/{512}})}\leq C e^{- \hat{C} (|\alpha|+\sqrt{|\lambda|})}.
\label{abtra}
\end{equation}
Since the doubling estimate  on the half ball has been shown in (\ref{halfdouble}), using the doubling inequality finitely many times, we obtain that
\begin{align}
\|\bar u \|_{L^2(\mathbb B^+_{1/{512}})}&\geq e^{- C (|\alpha|+\sqrt{|\lambda|})} \|\bar u \|_{L^2(\mathbb B^+_{1/2})}\nonumber \\
&\geq e^{-C (|\alpha|+\sqrt{|\lambda|})},
\end{align}
which contradicts the condition (\ref{abtra}) since $\hat{C}$ is an arbitrary constant that can be chosen to be sufficiently large. Thus, the condition (\ref{haoma}) holds.

Next we claim that there exists a constant $C$ such that
\begin{equation}
\|\bar u \|_{L^2( {\bf B}_{1/{5}})}\geq  e^{-C(|\alpha|+\sqrt{|\lambda|})}.
\label{claim}
\end{equation}
We recall the following interpolation inequality in \cite{R} or \cite{BL}. For any small constant $0<\epsilon <1$, there holds
\begin{align}
\|\nabla' w \|_{L^2(\mathbb R^{n-1})} \leq \epsilon^{\frac{3}{2}} (\|\nabla \nabla' w\|_{L^2(\mathbb R^{n}_+)} +\| w\|_{L^2(\mathbb R^{n}_+)})
+ \epsilon^{-\frac{1}{3}} \|w\|_{L^2(\mathbb R^{n-1})},
\label{interpo}
\end{align}
where $\nabla'$ is the derivative for first $n-1$ variables.
We choose $w$ to be $\bar u \eta$, where $\eta$ is a radial cut-off function such that $\eta=1$ in $\mathbb B^+_{1/6}$ and vanishes outside $\mathbb B^+_{1/5}$. Substituting $ w=\bar u  \eta$ in the interpolation inequality  (\ref{interpo}) gives that
\begin{align}
\|\nabla' (\bar u \eta) \|_{L^2(\mathbb R^{n-1})} \leq \epsilon^{\frac{3}{2}} \big(\|\nabla \nabla' ( \bar u \eta) \|_{L^2(\mathbb B^+_{1/5})} +\|  \bar u \eta \|_{L^2(\mathbb B^+_{1/5})}\big)
+ \epsilon^{-\frac{1}{3}} \|\bar u \eta \|_{L^2(  {\bf B}_{1/5})},
\label{mamazel}
\end{align}
Using the fact that $g_{in}=0$ for $i\not=n$ and $\frac{\partial \bar u }{\partial \nu}=0$ on $\{x_n=0\}$, the following Caccioppolli inequality holds,
\begin{align}
 \|\nabla \bar u \|_{L^2(\mathbb B_r^+)} \leq \frac{C(|\alpha|+\sqrt{|\lambda|})}{r}  \| \bar u \|_{L^2(\mathbb B_{2r}^+)}.
 \label{cacc}
\end{align}
Applying the estimates (\ref{cacc}) for second order derivative of $\bar u$ in (\ref{mamazel}), we derive that
\begin{align}
\|\nabla' \bar u   \|_{L^2( {\bf B}_{1/6})} \leq C_0\epsilon^{\frac{3}{2}} C( |\alpha|, \ \sqrt{|\lambda|}) +\epsilon^{-\frac{1}{3}} \|\bar u \|_{L^2( {\bf B}_{1/5})},
\end{align}
where we have used (\ref{normalize}) and $C( |\alpha|, \ \sqrt{|\lambda|})$ is a constant with polynomial growth of  $\alpha$ and $\sqrt{|\lambda|}$.
Adding $\|\bar u  \|_{L^2( {\bf B}_{1/6})}$ to both sides of the last inequality yields that
\begin{align}
\| \bar u   \|_{H^1 ( {\bf B}_{1/6})} \leq C_0 \epsilon^{\frac{3}{2}} C( |\alpha|, \ \sqrt{|\lambda|}) + 2\epsilon^{-\frac{1}{3}} \|\bar u \|_{L^2( {\bf B}_{1/5})}.
\label{hardlife}
\end{align}

To incorporate the first term on the right hand side of the last inequality into the left hand side, we choose $\epsilon$ such that
\begin{equation}
C_0 \epsilon^{\frac{3}{2}}  C( |\alpha|, \ \sqrt{|\lambda|})=\frac{1}{2} \| \bar u  \|_{H^1 ( {\bf B}_{1/6})}.
\end{equation}
That is,
$$ \epsilon=(\frac{\| \bar u  \|_{H^1 ( {\bf B}_{1/6})}}{2 C_0 C( |\alpha|, \ \sqrt{|\lambda|})})^{-2/3}. $$
Therefore, (\ref{hardlife}) turns into
\begin{align}
\|\bar u  \|_{H^1 ( {\bf B}_{1/6})}^{11/9} \leq 4 \big(2 C_0 C( |\alpha|, \ \sqrt{|\lambda|})\big)^{2/9} \|\bar u \|_{L^2( {\bf B}_{1/5})}.
\end{align}

Because of (\ref{haoma}), we infer that
\begin{align}
\| \bar u   \|_{H^1 ( {\bf B}_{1/6})} \leq e^{ C(|\alpha|+\sqrt{|\lambda|})} \|\bar u \|_{L^2( {\bf B}_{1/5})}.
\label{epsilon}
\end{align}
From (\ref{haoma}) again, it also follows that
\begin{align}
\|\bar u \|_{L^2( {\bf B}_{1/5})}\geq e^{ -C(|\alpha|+\sqrt{|\lambda|})},
\label{redouble}
\end{align}
which verifies the claim (\ref{claim}).

Let $\bar \eta$ be a cut-off function  such that $\bar \eta(y)=1$ for $|y|\leq \frac{1}{4}$ and vanishes for $|y|\geq \frac{1}{3}$.  By the Hardy trace inequality and elliptic estimates (\ref{cacc}), it follows that
\begin{align}
\|\bar u \|_{L^2( {\bf B}_{1/4})} &\leq \|\bar \eta \bar u \|_{L^2( {\bf B}_{1/4})}  \leq \|\nabla( \bar \eta \bar u )\|_{L^2( \mathbb R^{n}_+)} \nonumber \\
&\leq C \|\nabla \bar u \|_{L^2( {\mathbb B}^+_{1/3})}+ C \|\bar u \|_{L^2( {\mathbb B}^+_{1/3})}\nonumber \\
&\leq C(|\alpha|+\sqrt{|\lambda|}) \| \bar u \|_{L^2( {\mathbb B}^+_{1/2})}\nonumber \\
&\leq C(|\alpha|+\sqrt{|\lambda|}).
\label{lastt}
\end{align}
Combining the established estimates (\ref{redouble}) and (\ref{lastt}), we have
\begin{equation}
\|\bar u \|_{L^2( {\bf B}_{1/4})}\leq e^{ C(|\alpha|+\sqrt{|\lambda|})} \|\bar u \|_{L^2( {\bf B}_{1/5})}.
\label{night}
\end{equation}
Notice that $u=\bar u $ on ${\bf B}_{1/2}$. By rescaling and diffeomorphism of Fermi exponential map, we arrive at
\begin{align}
\| u \|_{L^2( {\mathbb B}_{2r}(x_0))}  \leq e^{ C(|\alpha|+\sqrt{|\lambda|})} \| u \|_{L^2( {\mathbb B}_{r}(x_0))}
\label{double}
\end{align}
for any $x_0\in \partial\Omega$, ${\mathbb B}_{2r}(x_0)\subset \partial\Omega$, and $r<r_0$ for some $r_0$ depending only on $\partial\Omega$. This completes the proof of Theorem \ref{th3}.

\end{proof}

We will show the upper bounds of nodal sets for Neumann and Robin eigenfunction on the analytic boundary. To achieve it, we need a quantitative inequality on the relation of $L^2$ norm of eigenfunctions on the boundary and on the half balls.
We argue on scale with $\delta=1$ for equations (\ref{half}) with the conditions (\ref{halfcon}).
 Applying quantitative two half-ball and one lower dimensional ball in (\ref{tool}) by replacing $u$ by $\bar u$, and the doubling inequalities in the half ball in (\ref{halfdouble}) finitely many times, we have
\begin{align}
\| \bar u \|_{L^2 (\mathbb B^+_{1/2})} \leq  e^{C(|\alpha|+\sqrt{|\lambda|})}\| \bar u \|^\kappa_{L^2 (\mathbb B^+_{1/2})} \| \bar u \|^{1-\kappa}_{H^1 ({\bf B}_{1/3})}.
\end{align}
Thus, we obtain that
\begin{align}
\| \bar u \|_{L^2 (\mathbb B^+_{1/8})} \leq  e^{C(|\alpha|+\sqrt{|\lambda|})} \|\bar u \|_{H^1 ({\bf B}_{1/3})}.
\label{impro}
\end{align}

By the arguments in deriving the estimates (\ref{epsilon}), we can improve (\ref{impro}) as
\begin{align}
\|\bar u \|_{L^2 (\mathbb B^+_{1/8})} \leq  e^{C(|\alpha|+\sqrt{|\lambda|})} \| \bar u \|_{L^2 ({\bf B}_{2/5})}.
\label{rescal}
\end{align}

We are ready to provide the proof of upper bounds for the boundary nodal sets of Neumann eigenfunctions.
\begin{proof}[Proof of theorem \ref{thnew}] To get the measure estimates of nodal sets of Neumann eigenfunctions on the boundary, we perform a standard lifting argument. Let
\begin{align}
\hat{w}(x, t)=e^{\sqrt{\lambda} t} u(x).
\label{lift}
\end{align}
Then $\hat{w}(x, t)$ satisfies the following equation
\begin{equation}
\left \{ \begin{array}{rll}
-\hat{\triangle}  \hat{w}=0 \quad &\mbox{in} \ {\Omega}\times (-\infty, \ \infty),  \medskip\\
\frac{\partial\hat{w}}{\partial \nu}=0 \quad &\mbox{on} \ {\partial\Omega}\times (-\infty, \ \infty)
\end{array}
\right.
\label{error}
\end{equation}
with $\hat{\triangle}=\triangle +\partial^2_t$. By straightening the boundary $\partial\Omega$ locally, rescaling and translation. we may assume that $(p, 0, t)\in \big(\partial \mathbb B^+_{1/16}\cap \{x_n=0\}\big)\times (-\frac{1}{16}, \frac{1}{16}) $ with $p\in \mathbb R^{n-1}.$  From elliptic estimates in Lemma 2.3 in \cite{MN} or Corollary 7.2 in \cite{R},  we obtain that
\begin{align}
| \frac{\nabla^{'\bar\alpha} \hat{w} (p, 0, 0)}{ \bar\alpha !} |\leq C \hat{C}^{k}  \|\hat{w}   \|_{L^\infty \big(\mathbb B^+_{1/8}\times(-\frac{1}{8}, \frac{1}{8})\big)},
\end{align}
where
$ |\bar \alpha|=k$, the derivative is taken with respect to $x'$ in $\partial \mathbb B^+_{1/16}\cap \{x_n=0\}$, and $\hat{C}>1$ depends on $\Omega$.
By the definition of $\hat{w}$, we have that
 \begin{align}
| \frac{\nabla^{'\bar\alpha} {u} (p, 0)}{ \bar\alpha !} |\leq C \hat{C}^{k} e^{C\sqrt{\lambda}} \|u   \|_{L^\infty (\mathbb B^+_{1/8})}.
\end{align}
Then $\bar u (p, 0)$ is real analytic for any $(p, 0)\in \partial \mathbb B^+_{1/16}\cap\{x_n=0\}$. We may consider $p$ as the origin in $\mathbb R^{n-1}$. Summing up a geometric series gives a holomorphic extension of ${ u}$ with
\begin{equation}
\sup_{ |z|\leq  \frac{1}{2\hat{C}} }|  u (z)|\leq e^{C\sqrt{\lambda}} \| u   \|_{L^\infty (\mathbb B^+_{1/8})},
\label{zeld}
\end{equation}
where $\frac{1}{2\hat{C}}<\frac{1}{8}$ and $z\in \mathbb C^{n-1}$. The estimates (\ref{rescal}) also hold for Neumann boundary conditions with $\alpha=0$. Hence,
it follows that
\begin{align}
\|u \|_{L^2 (\mathbb B^+_{1/8})} \leq  e^{C\sqrt{\lambda}} \|  u \|_{L^2 ({\bf B}_{2/5})}.
\end{align}
Note that ${\bf B}_r$ is denoted as the ball in $\mathbb R^{n-1}$ with radius $r$.
Taking the boundary doubling inequality (\ref{night}) with $\alpha=0$, (\ref{zeld}) and elliptic estimates into consideration, by finite steps iterations,
we conclude that
\begin{align}
\sup_{ |z|\leq  \frac{1}{2\hat{C}}}|  u (z)|\leq e^{C\sqrt{\lambda}}\sup_{ x\in {\bf B}_{\frac{1}{4\hat{C}}}}|  u (x)|.
\end{align}
By rescaling arguments, we derive that
\begin{align}
\sup_{ |z|\leq 2 r}|  u (z)|\leq e^{C\sqrt{\lambda}}\sup_{ x\in  {\bf B}_r}|  u (x)|,
\label{eventu}
\end{align}
where $0<r<\hat{r}_0$ and $\hat{r}_0$, $C$ depends on $\Omega$.

Thanks to the doubling inequality (\ref{eventu}) and the growth control lemma for zeros, i.e. Lemma \ref{wwhy}, we are ready to give the proof of Theorem \ref{thnew}. Since $r$ does not depend on $\lambda$,
we can argue on scales of $r=1$. Let $p\in {\bf B}_{1/4}\subset \mathbb R^{n-1}$ be the point where the supremum of $| u |$ is achieved. After rescaling, we assume that $|u(p)|=1$.
For each direction $\omega\in S^{n-2}$, we consider the function $$ u_\omega(z)= u (p+z\omega), \quad \quad z\in \mathcal{B}_1(0)\subset \mathbb{C}.$$ By the doubling inequality (\ref{eventu}) and Lemma \ref{wwhy}, we obtain that
\begin{align}
&\sharp\{ x\in {\bf B}_{1/2}(p)\subset \mathbb R^{n-1} | x-p \ \mbox{is parallel to} \ \omega \ \mbox{and} \  u (x)=0\} \nonumber \\
&\leq \sharp\{ z\in \mathcal{B}_{1/2}(0)\subset \mathbb{C}| u_\omega(z)=0\}  \nonumber \\
&\leq C\sqrt{\lambda}.
\end{align}
By the integral geometry estimates, we further derive that
\begin{align}
H^{n-2} \{ x\in {\bf B}_{1/2}(p)|  u (x)=0\} &\leq \int_{S^{n-2}} C\sqrt{\lambda}\, d\omega \nonumber\\
&\leq C\sqrt{\lambda}.
\end{align}
Thus, we show the upper bound of nodal sets
\begin{align}
H^{n-2} \{ x\in {\bf B}_{1/4}(0)|  u (x)=0\} \leq C\sqrt{\lambda}.
\end{align}
By rescaling, it also implies that
\begin{align}
H^{n-2} \{ \mathbb B_{r_0}(p)\subset \partial\Omega | u(x)=0\} \leq C\sqrt{\lambda}.
\end{align}
for some $r_0$ depending only on $\Omega$ and for any $p\in \Omega$.
 Since the boundary $\partial\Omega$ is compact, by finite coverings, the theorem is arrived.
\end{proof}

By the strategy in the proof of Theorem \ref{thnew},
we consider the boundary nodal sets of Robin eigenfunctions (\ref{robin}).
\begin{proof}[Proof of Corollary \ref{th2}]
As in section 4, using lifting arguments, we first get rid of $\lambda$ and $\alpha$. Let
\begin{align*}
\hat{w}(x, t, s)=e^{\alpha t} e^{(|\alpha|i+\sqrt{\lambda})s} u(x).
\end{align*}
Applying Fermi coordinates to flatten the boundary and the rescaling arguments, we have the following equation locally
\begin{equation}
\left \{ \begin{array}{rll}
-\hat{\triangle}  \hat{w}=0 \quad &\mbox{in} \ \mathbb B^+_1\times (-\infty, \ \infty)\times(-\infty, \ \infty) ,  \medskip\\
\frac{\partial\hat{w}}{\partial x_n}+\frac{\partial\hat{w}}{\partial t}=0 \quad &\mbox{on} \ \partial\mathbb B^+_1\cap\{x_n=0\}\times (-\infty, \ \infty)\times (-\infty, \ \infty)
\end{array}
\right.
\end{equation}
where $\hat{\triangle}=\triangle +\partial^2_t+\partial_s^2$. By Cauchy-Kovalesvsky theorem in \cite{H}, we can extend $\hat{w}$ analytically across the boundary $\partial\Omega\times [-1, \ -1] \times [-1, \ 1]$ in $\widehat{\Omega}_1 \times [-1, \ -1] \times [-1, \ 1]$, where $\widehat{\Omega}_1=\{x\in \mathbb R^n| \dist(x, \ \Omega)\leq \delta\}$ and $0<\delta<\frac{1}{24}$ depending only on $\Omega$. Choose any point
$(p, 0)\in \partial\mathbb B^+_{\frac{1}{16}}\cap \{x_n=0\}$.
Using elliptic estimates for $-\hat{\triangle}  \hat{w}=0$, we have
\begin{align}
| \frac{\nabla^{'\bar\alpha} \hat{w} (p, 0, 0, 0)}{ \bar\alpha !} |&\leq C \hat{C}^{k}  \|\hat{w}   \|_{L^\infty \big(\mathbb B_{\delta}(p, 0)\times(-\delta, \delta)\times(-\delta, \delta)\big)} \nonumber \\
&\leq C \hat{C}^{k}  \|\hat{w}   \|_{L^\infty \big(\mathbb B^+_{\frac{1}{8}}\times(-\frac{1}{8}, \frac{1}{8})\times(-\frac{1}{8}, \frac{1}{8})\big)},
\label{kaoz}
\end{align}
where $|\bar\alpha|=k$ and $\hat{C}>1$ depends on $\Omega$. In the second inequality of (\ref{kaoz}), we apply the growth control estimate as (\ref{agree}) from the analytic continuation. From the definition of $\hat{w}(x, t, s)$, we derive that
\begin{align}
| \frac{\nabla^{'\bar\alpha} u (p, 0)}{ \bar\alpha !} |
&\leq C \hat{C}^{k} e^{C(\alpha+\sqrt{|\lambda|})} \|u  \|_{L^\infty(\mathbb B^+_{\frac{1}{8}})}.
\label{kizel}
\end{align}
The rest of the proof follows from Theorem 3 by using the boundary doubling inequality (\ref{night}) and Lemma \ref{wwhy}. Based on bounds in (\ref{kizel}), we can show the upper bounds of boundary nodal sets in the Corollary.
\end{proof}

At last, we show the sharpness of the upper bound of boundary nodal sets for Neumann eigenfunctions.
\begin{remark}
The upper bound for boundary nodal sets of Neumann eigenfunctions in (\ref{Sha}) is sharp.
\end{remark}
 We first consider the Neumann eigenvalue problem (\ref{Neumann}) in a disc with radius one in $\mathbb R^2$. By separation of variables, we can write the eigenfunction  as $u(x)=R(r) \Phi(\theta)$. Direct calculations show that
\begin{equation}
R''(r)+\frac{1}{r}R'(r)+(\lambda-\frac{k^2}{r^2})R(r)=0
\label{bess1}
\end{equation}
and
\begin{equation}
-\Phi''(\theta)=k^2\Phi(\theta),   \quad \quad k=1, 2, \cdots.
\end{equation}
Let $y=\sqrt{\lambda}r$ and $J(y)= R(\frac{y}{\sqrt{\lambda}})$. It follows that
\begin{equation}
y^2 J''(y)+y J'(y)+(y^2-k^2)J(y)=0,
\end{equation}
 which is the well-known Bessel's equation. From the  Neumann boundary condition on $r=1$, we derive that $R'(1)=J'(\sqrt{\lambda})=0$.
 Let $J_k(y)$ be the $k$th Bessel function and $\alpha_{k,j}$ be the $j$th zeros of $J'_k(y)$.
 Then the eigenfunctions are given by $J_k(\alpha_{k,j}r)\sin(k\theta)$ or $J_k(\alpha_{k,j}r)\cos(k\theta)$ and the eigenvalues $\lambda=\alpha_{k,j}^2$.  It is known in \cite{O} that $\alpha_{k,j}\approx k+\delta_j k^\frac{1}{3}+O(k^{-\frac{1}{3}})$ for large $k$ and fixed $j$, where $\delta_j $ is some known constant depending on $j$. On the boundary of the disc with $r=1$. There are at most $k$ nodal points. From the asymptotic behavior of eigenvalue $\alpha_{k,j}$, we learn that the conclusion (\ref{Sha}) is sharp in $n=2$.

For $n\geq 3$, we  consider again the Neumann eigenvalue problem in a ball with radius 1. By separation of variables, let $u(x)= R(r)\Phi(\omega)$,
Then $R(r)$ satisfies the equations
\begin{equation}
R''(r)+\frac{n-1}{r}R'(r)+(\lambda-\frac{\gamma}{r^2})R(r)=0
\label{besse}
\end{equation}
and
\begin{equation}
-\triangle_\omega \Phi=\gamma \Phi \quad \quad \mbox{on} \ S^{n-1},
\end{equation}
where $\gamma=k(k+n-2)$ is the eigenvalue for the spherical harmonics on $S^{n-1}$. By a standard scaling, let $W(r)=r^{\frac{n-2}{2}} R(r)$.
Equation (\ref{besse}) is reduced to the equation
\begin{equation}
W''(r)+\frac{1}{r}W(r)+\big(\lambda-(\gamma+\frac{(n-2)^2}{4})/{r^2}  \big)W(r)=0.
\end{equation}
Thus, as in the dimension $n=2$, we can write the solutions as
$$W(r)=J_{\sqrt{\gamma+\frac{(n-2)^2}{4}}}(\sqrt{\lambda}r),$$
where $J_{\sqrt{\gamma+\frac{(n-2)^2}{4}}}(y)$ is the Bessel function.
That is,
$$ R(r)= \frac{J_{\sqrt{\gamma+\frac{(n-2)^2}{4}}}(\sqrt{\lambda}r)}{r^{\frac{n-2}{2}}}. $$
The Neumann boundary condition $$\frac{\partial u}{\partial \nu}=\frac{\partial (R(r)\Phi(\omega))}{\partial r}=0$$ on $r=1$ implies that
\begin{equation}
-\frac{n-2}{2}J_{\sqrt{\gamma+\frac{(n-2)^2}{4}}}(\sqrt{\lambda})+J'_{\sqrt{\gamma+\frac{(n-2)^2}{4}}}(\sqrt{\lambda})\sqrt{\lambda}=0.
\label{bess}
\end{equation}
The measure of nodal sets for spherical harmonics $\Phi(\omega)$ is known in \cite{DF} as
\begin{equation}
c\sqrt{\gamma}\leq H^{n-2}\{ \omega\in S^{n-1}|\Phi(\omega)=0\}\leq C\sqrt{\gamma}.
\label{sphere}
\end{equation}
Let $C_{\sqrt{\gamma+\frac{(n-2)^2}{4}}, \  k}$ be the $k$th positive zeros of the solution in (\ref{bess}), it is shown in Theorem 2.1 and Theorem 2.3 in \cite{ES} that $$C^2_{\sqrt{\gamma+\frac{(n-2)^2}{4}}, \ 1}\approx {\gamma} $$ as $\gamma$ is large. Let  $\sqrt{\lambda}=C_{\sqrt{\gamma+\frac{(n-2)^2}{4}}, \ 1}$. It follows from (\ref{sphere}) that the conclusion in the Theorem \ref{thnew} is optimal for $n\geq 3$.

\section{Global Carleman estimates}
In this section, we prove the quantitative global Carleman estimates in Proposition \ref{Pro4}. Interested readers may refer to the survey \cite{K} and \cite{LL}
for more exhaustive literature for local and global Carleman estimates.
We will use the integration by parts arguments repeatedly to get the desired estimates. Recall that
the weight function $$\psi(y)= e^{sh(y)} $$ with $$h(y)=-\frac{|y'|^2}{4}+\frac{y^2_n}{2}-y_n.$$
Actually, the weight function can be chosen as any $h\in C^2$ such that $|\nabla h|\not=0$ in $\mathbb B^+_{1/2}$ to have the Carleman estimates in Proposition \ref{Pro4}.
Recall  the assumptions about $b(y)$ and $c(y)$ are
\begin{equation}
\left\{ \begin{array}{lll}
\| {b}\|_{W^{1, \infty} (\mathbb B^+_{1/2})}\leq C(|\alpha|+1), \medskip \\
\| {c}\|_{W^{1, \infty}(\mathbb B^+_{1/2})}\leq C(\alpha^2 +|\lambda|).
\end{array}
\right.
\label{aab}
\end{equation}

\begin{proof}[Proof of Proposition \ref{Pro4}]
Choose
\begin{equation} w(y)= e^{\tau \psi(y)} v(y).
\label{wvv}\end{equation} Since $v(y)\in C^\infty (\mathbb B^+_{1/2})$, then $w(y)\in C^\infty (\mathbb B^+_{1/2})$. We introduce  a second order elliptic operator $$P_0=- a_{ij} D_{ij} +b_i(y) D_i +c(y).$$

Define the conjugate operator as
$$ P_\tau w= e^{\tau\psi(y)} P_0( e^{-\tau\psi(y)}w).  $$
Direct calculations show that
\begin{align}
P_\tau w=&-a_{ij}D_{ij}w+2\tau a_{ij}D_i \psi D_j w+ \tau a_{ij} D_{ij} \psi w \nonumber \\ & -\tau^2 a_{ij}D_i\psi D_j \psi w-\tau b_i(y) D_i\psi w+b_i(y)D_iw+c(y)w \nonumber \\
=&-a_{ij}D_{ij}w+2\tau s\psi a_{ij}D_i h D_j w-\tau^2 s^2\psi^2\beta(y) w+\tau\psi a(y,s)w\nonumber \\&-\tau s\psi b_i(y)D_i h w +b_i(y)D_iw+c(y)w,
\end{align}
where $$\beta(y)=a_{ij}D_i h D_j h,$$
\begin{equation} a(y,s)=s^2 \beta(y)+s a_{ij}D_{ij}h.
\label{axs}
\end{equation}
Note that $\beta(y)\geq C$ for some positive constant $C$ on $\mathbb B^+_{1/2}$ by the uniform ellipticity.
We split the expression $P_\tau w $ into the sum of two expressions $P_1 w$ and $P_2 w$, where
$$ P_1 w=-a_{ij} D_{ij}w-\tau^2 s^2\psi^2\beta(y) w- \tau s\psi b_i(y)D_i h w +c(y)w, $$
$$ P_2 w=2\tau s\psi a_{ij}D_i h D_j w+b_i(y)D_iw.  $$
Then
\begin{equation}
P_\tau w= P_1 w+ P_2 w+ \tau\psi a(y,s)w.
\end{equation}
We compute the $L^2$ norm of $P_\tau w$. By triangle inequality, we have
\begin{align}
\|P_\tau w\|^2&=\|P_1 w+P_2 w+\tau\psi a(y,s)w\|^2 \nonumber \\
&\geq \|P_1 w\|^2+ \|P_2 w\|^2+ 2 \langle P_1 w, \  P_2 w\rangle -\|\tau\psi a(y,s)w\|^2.
\label{keyinner}
\end{align}
Later on, we will absorb the term $\|\tau\psi a(y,s)w\|^2$.  Now we are going to derive a lower bound for the inner product in (\ref{keyinner}).
Let's write
\begin{equation}
\langle P_1 w, \  P_2 w\rangle=\sum^4_{k=1} I_k+\sum^4_{k=1} J_k,
\label{right}
\end{equation}
where
\begin{align} &I_1= \langle -a_{ij} D_{ij}w, \  2\tau s\psi a_{ij}D_i h D_j w    \rangle, \nonumber \\
 &I_2=  \langle -\tau^2 s^2\psi^2\beta(y) w, \      2\tau s\psi a_{ij}D_i h D_j w         \rangle,  \nonumber \\
 &I_3=  \langle  - \tau s\psi b_i(y)D_i h w, \   2\tau s\psi a_{ij}D_i h D_j w     \rangle,   \nonumber \\
& I_4= \langle c(y)w, \      2\tau s\psi a_{ij}D_i h D_j w\rangle,     \nonumber \\
 &J_1=  \langle -a_{ij} D_{ij}w, \   b_i(y)D_iw \rangle,  \nonumber \\
& J_2= \langle  -\tau^2 s^2\psi^2\beta(y) w, \   b_i(y)D_iw \rangle,   \nonumber \\
 &J_3=\langle  - \tau s\psi b_i(y)D_i h w, \  b_i(y)D_iw \rangle, \nonumber \\
 & J_4= \langle c(y)w, \   b_i(y)D_iw \rangle. \end{align}
We will estimate each term on the right hand side of (\ref{right}). Performing the integration by parts shows that
\begin{align}
I_1=&2\tau s^2 \int_{\mathbb B_{1/2}^+} \psi a_{ij}D_i wD_j h a_{kl} D_k h D_l w\,dy+2\tau s \int_{\mathbb B_{1/2}^+} D_j( a_{ij} a_{kl} D_k h) \psi D_iw
D_lw \,dy \nonumber \\ &+2\tau s\int_{\mathbb B_{1/2}^+}\psi a_{ij}D_i w a_{kl} D_k h D_{lj} w\,dy-2\tau s\int_{\partial \mathbb B_{1/2}^+} \psi a_{kl} D_k h D_l w a_{ij} D_iw \nu_j \,dS\nonumber \\
=&I_1^{1}+I_1^2+ I_1^3+I_1^4.
\label{lala}
\end{align}
The first term $I_1^{1}$ can be controlled as
\begin{align}
I_1^{1}&=2\tau s^2 \int_{\mathbb B_{1/2}^+}\psi | a_{ij} D_iw D_j h|^2 \, dy \nonumber \\ \label{lala1}
&\geq 0.
\end{align}
Applying the integration by parts argument, the third term $I_1^3$ can be computed as
\begin{align}
I_1^3&=\tau s \int_{\mathbb B_{1/2}^+}\psi  a_{ij}a_{kl}D_k h D_l (D_i w D_j w)\,dy \nonumber  \\
&=-\tau s\int_{\mathbb B_{1/2}^+} \psi D_l (a_{ij}a_{kl}D_k h) D_i w D_j w\,dy-\tau s^2 \int_{\mathbb B_{1/2}^+}\psi \beta(y) a_{ij} D_i w D_j w\,dy \nonumber  \\
&+\tau s\int_{\partial \mathbb B_{1/2}^+}\psi  a_{ij} D_i w D_j w a_{kl} D_k h \nu_l dS.
\label{lala2}
\end{align}
Combining (\ref{lala}), (\ref{lala1}) and (\ref{lala2}), and using the fact that $\|a_{ij}\|_{C^1}$ is bounded,   we can estimate $I_1$ from below
\begin{align}
I_1&\geq 2\tau s^2 \int_{\mathbb B_{1/2}^+}\psi | a_{ij} D_iw D_j h|^2 \, dy-\tau s^2 \int_{\mathbb B_{1/2}^+}\psi \beta(y)  a_{ij} D_i w D_j w\,dy \nonumber \\
&-C\tau s  \int_{\mathbb B_{1/2}^+}\psi |\nabla w|^2\, dy-C\tau s\int_{\partial \mathbb B_{1/2}^+}\psi |\nabla w|^2\, dy.
\end{align}
Thus,
\begin{align}
I_1 \geq-\tau s^2 \int_{\mathbb B_{1/2}^+}\psi \beta  a_{ij} D_iw D_j w \, dy- C\tau s  \int_{\mathbb B_{1/2}^+}\psi |\nabla w|^2\, dy-C\tau s\int_{\partial \mathbb B_{1/2}^+}\psi |\nabla w|^2\, dy.
\label{II1}
\end{align}

Now we compute the term $I_2$ using integration by parts argument,
\begin{align}
I_2&=-\tau^3s^3 \int_{\mathbb B_{1/2}^+}\psi^3 \beta(y) a_{ij}D_i h D_j w^2 \, dy \nonumber  \\
&=3 \tau^3 s^4 \int_{\mathbb B_{1/2}^+}\psi^3 \beta(y)  a_{ij} D_i h D_j h w^2 \, dy +\tau^3 s^3 \int_{\mathbb B_{1/2}^+} \psi^3 D_j (\beta a_{ij} D_i h) w^2 \, dy \nonumber  \\
&-\tau^3 s^3 \int_{\partial \mathbb B_{1/2}^+}\psi^3 \beta(y) w^2 a_{ij} D_i h \nu_j \, dS \nonumber \\
&\geq 3 \tau^3 s^4 \int_{\mathbb B_{1/2}^+}\psi^3 \beta^2(y)  w^2 \, dy- C\tau^3 s^3  \int_{\mathbb B_{1/2}^+}\psi^3   w^2 \, dy-C\tau^3 s^3  \int_{\partial \mathbb B_{1/2}^+}\psi^3   w^2 \, dS.
\end{align}
Choosing $s$ large enough and noting that $\beta(y)\geq C$, we deduce that
\begin{align}
I_2\geq \frac{17}{6} \tau^3 s^4 \int_{\mathbb B_{1/2}^+}\psi^3 \beta^2 w^2 \, dy-C\tau^3 s^3  \int_{\partial \mathbb B_{1/2}^+}\psi^3   w^2 \, dS.
\label{II2}
\end{align}

For the term $I_3$, using the integration by parts argument leads to
\begin{align}
I_3&=-\tau^2 s^2\int_{\mathbb B_{1/2}^+} \psi^2 b_k(y) D_k h a_{ij} D_i h D_j w^2\,dy \nonumber \\
&=\tau^2 s^2 \int_{\mathbb B_{1/2}^+} D_j(\psi^2 b_k(y) D_k h a_{ij} D_i h) w^2\,dy- \tau^2 s^2 \int_{\partial \mathbb B_{1/2}^+}
\psi^2w^2 b_k(y) D_k h a_{ij} D_i h\nu_j \,dS.
\end{align}
Making use of the assumption of (\ref{aab}) gives that
\begin{align}
I_3\geq -C(|\alpha|+1) \tau^2 s^3 \int_{\mathbb B_{1/2}^+} \psi^2 w^2\,dy-C(|\alpha|+1) \tau^2 s^2 \int_{\partial \mathbb B_{1/2}^+}
\psi^2w^2 \,dS.
\label{II3}
\end{align}

We proceed to estimate the term $I_4$. Using integration by parts shows that
\begin{align}
I_4 &=\tau s \int_{ \mathbb B_{1/2}^+} c(y) \psi  a_{ij} D_i h D_j w^2 \,dy \nonumber\\
& =\tau s \int_{ \mathbb B_{1/2}^+} \psi a_{ij} D_jc(y)   D_i h  w^2  \,dy-\tau s^2 \int_{ \mathbb B_{1/2}^+} c(y) \psi
a_{ij} D_j h   D_i h w^2\,dy\nonumber\\
&-\tau s \int_{ \mathbb B_{1/2}^+}c(y)\psi  D_j( a_{ij} D_i h)w^2\,dy+\tau s \int_{\partial \mathbb B_{1/2}^+} c(y) \psi w^2 a_{ij} D_i h \nu_j \, dS.
\end{align}

Again, the assumptions of (\ref{aab}) leads to
\begin{align}
I_4\geq -C\tau s^2(\alpha^2+|\lambda|) \int_{ \mathbb B_{1/2}^+}\psi w^2\,dy -C\tau s(\alpha^2+|\lambda|) \int_{ \partial \mathbb B_{1/2}^+}\psi w^2\,dy.
\label{II4}
\end{align}

Together with the estimates on each $I_k$ from (\ref{II1}) to (\ref{II4}), using the assumption that $\tau> C_s (|\alpha|+\sqrt{|\lambda|})$ for some $C_s$ depending on $s$, we arrive at
\begin{align}
\sum^4_{k=1}I_k&\geq \frac{14}{5} \tau^3 s^4 \int_{\mathbb B_{1/2}^+}\psi^3\beta^2 w^2 \, dy-C\tau^3 s^3  \int_{\partial \mathbb B_{1/2}^+}\psi^3   w^2 \, dS-C\tau s
\int_{\partial \mathbb B_{1/2}^+}\psi |\nabla w|^2 \, dS  \nonumber \\
&-\tau s^2 \int_{\mathbb B_{1/2}^+}\psi \beta  a_{ij} D_iw D_j w \, dy- C\tau s  \int_{\mathbb B_{1/2}^+}\psi |\nabla w|^2\, dy.
\label{III}
\end{align}

Next we continue to estimate the integration about $J_k$ using the strategy of integration by parts. Direct computations show that
\begin{align}
J_1&=\int_{ \mathbb B_{1/2}^+} D_j(a_{ij} b_k) D_i w D_k w\,dy+ \frac{1}{2} \int_{ \mathbb B_{1/2}^+} a_{ij} b_k D_k( D_i w D_j w) \,dy \nonumber \\
&-\int_{ \partial \mathbb B_{1/2}^+} b_k D_k w a_{ij} D_i w \nu_j\, dS \nonumber \\
& =\int_{ \mathbb B_{1/2}^+} D_j(a_{ij} b_k) D_i w D_k w\,dy- \frac{1}{2} \int_{ \mathbb B_{1/2}^+} D_k (a_{ij} b_k) D_i w D_j w \,dy \nonumber \\
&+\frac{1}{2} \int_{ \partial \mathbb B_{1/2}^+} a_{ij} D_i w D_j w b_k \nu_k \,dS-  \int_{ \partial \mathbb B_{1/2}^+}b_k D_k w  a_{ij}  D_i w \nu_j \,dS.
\end{align}
Thus, from the assumption of $b_i$,
\begin{align}
J_1 \geq-C(|\alpha|+1)\int_{ \mathbb B_{1/2}^+} |\nabla w|^2\,dy-C(|\alpha|+1) \int_{ \partial \mathbb B_{1/2}^+}|\nabla w|^2 \,dS.
\label{JJ1}
\end{align}

For the term $J_2$, integration by parts argument yields that
\begin{align}
J_2&=-\frac{1}{2}\tau^2 s^2\int_{ \mathbb B_{1/2}^+} \psi \beta(y) b_i(y) D_i w^2 \,dy \nonumber \\
&=\frac{\tau^2 s^3}{2} \int_{ \mathbb B_{1/2}^+} \psi \beta(y) b_i D_i h w^2 \,dy+\frac{\tau^2 s^2}{2} \int_{ \mathbb B_{1/2}^+} \psi  D_i  \beta(y) b_i(y)w^2\,dy \nonumber \\
&+\frac{\tau^2 s^2}{2}\int_{ \mathbb B_{1/2}^+} \psi \beta(y)  D_i b_i(y) w^2 \,dy-\frac{\tau^2 s^2}{2} \int_{ \partial \mathbb B_{1/2}^+} \psi \beta(y) w^2 b_i(y) \nu_j \,d S.
\end{align}
Therefore, we can show that
\begin{align}
J_2\geq -C\tau^2 s^3 (|\alpha|+1) \int_{ \mathbb B_{1/2}^+} \beta\psi  w^2 \,dy - C\tau^2 s^2 (|\alpha|+1) \int_{\partial \mathbb B_{1/2}^+} \psi  w^2 \,dS.
\label{JJ2}
\end{align}

In the same way, we can show that
\begin{align}
J_3 &=-\frac{\tau s}{2} \int_{ \mathbb B_{1/2}^+} \psi b_k D_k h b_i D_i w^2\,dy \nonumber \\
&=\frac{\tau s}{2}  \int_{ \mathbb B_{1/2}^+} D_i ( \psi b_k D_k h b_i) w^2\,dy -\frac{\tau s}{2}\int_{ \partial \mathbb B_{1/2}^+}  \psi w^2
b_k D_k h b_i \nu_j \,d S.
\end{align}

We can control $J_3$ below as
\begin{align}
J_3 \geq -C\tau s^2 (|\alpha|+1)^2 \int_{ \mathbb B_{1/2}^+} \psi  w^2\,dy -C\tau s (|\alpha|+1)^2  \int_{ \partial \mathbb B_{1/2}^+}  \psi w^2\,d S.
\label{JJ3}
\end{align}

Similarly, applying the integration by parts leads to
\begin{align}
J_4 &=\frac{1}{2} \int_{ \mathbb B_{1/2}^+} c(y)  b_i D_i w^2\,dy \nonumber \\
&=-\frac{1}{2}\int_{ \mathbb B_{1/2}^+} D_ic(y) b_i w^2\,dy-\frac{1}{2}\int_{ \mathbb B_{1/2}^+} c(y) D_i b_i w^2 \,dy \nonumber \\
&
+\frac{1}{2} \int_{ \partial \mathbb B_{1/2}^+} c(y) w^2 b_i \nu_i \, dS.
\end{align}
Then we can obtain that
\begin{align}
J_4 \geq -C(\alpha^2+|\lambda|)(|\alpha|+1) \int_{ \mathbb B_{1/2}^+}  w^2 \,dy -C(\alpha^2+|\lambda|)(|\alpha|+1) \int_{\partial \mathbb B_{1/2}^+}  w^2\, dS.
\label{JJ4}
\end{align}

Using the fact that $\tau >  C_s(|\alpha|+\sqrt{|\lambda|})$ for $C_s$ depending on $s$, and  summing up the estimates from (\ref{JJ1}) to (\ref{JJ4}) gives that
\begin{align}
\sum^4_{k=1} J_k&\geq -C(|\alpha|+1) \int_{ \mathbb B_{1/2}^+}\psi |\nabla w|^2\,dy- C\tau^2 (|\alpha|+1) s^3 \int_{ \mathbb B_{1/2}^+} \psi w^2\,dy \nonumber \\
& -C\tau^2 s^2 (|\alpha|+1)  \int_{ \partial \mathbb B_{1/2}^+}\psi w^2\,dS-C (|\alpha|+1)  \int_{ \partial \mathbb B_{1/2}^+}|\nabla w|^2\,dS.
\label{JJJ}
\end{align}

Recall the inner product (\ref{right}). Combining (\ref{III}), (\ref{JJJ}) and using the the assumption $\tau >  C_s(|\alpha|+\sqrt{|\lambda|})$ again, we derive that
\begin{align}
\langle P_1 w, \  P_2 w\rangle& \geq  \frac{11}{4}\tau^3 s^4 \int_{\mathbb B_{1/2}^+}\psi^3\beta^2 w^2 \, dy-\frac{5}{4}\tau s^2 \int_{ \mathbb B_{1/2}^+}\psi a_{ij} D_i w D_j w\,dy\nonumber \\
&-
C\tau^3 s^3  \int_{\partial \mathbb B_{1/2}^+}\psi^3   w^2 \, dS-C\tau s
\int_{\partial \mathbb B_{1/2}^+}\psi |\nabla w|^2 \, dS.
\label{inner}
\end{align}

We want to control the gradient term on the second term on the right hand side of (\ref{inner}). To this end, we consider the following inner product
\begin{equation}
\langle P_1 w, \ \tau s^2\psi\beta(y) w\rangle =\sum^4_{k=1} L_k,
\label{product}
\end{equation}
where
\begin{align}
&L_1=\langle -a_{ij}D_{ij}w, \ \tau s^2 \psi \beta w\rangle,\nonumber\\
&L_2=\langle -\tau^2 s^2\psi^2 \beta w, \tau s^2\psi\beta w\rangle=-\tau^3 s^4\int_{ \mathbb B_{1/2}^+}\psi^3 \beta^2 w^2\,dy,
\label{LL2}
\end{align}
\begin{align}
L_3&=\langle -\tau s\psi  b_i D_i h w,\ \tau s^2\psi \beta w\rangle\nonumber\\
&=-\tau^2 s^3\int_{ \mathbb B_{1/2}^+}\psi^2\beta b_i D_i hw^2\,dy\nonumber\\
&=-C\tau^2 s^3(|\alpha|+1) \int_{ \mathbb B_{1/2}^+}\psi^2 w^2\,dy,
\label{LL3}
\end{align}
and
\begin{align}
L_4&=\langle c(y) w,\ \tau s^2 \psi\beta w  \rangle\nonumber\\
&=\tau s^2 \int_{ \mathbb B_{1/2}^+}c(y)\psi \beta w^2\,dy\nonumber\\
&\geq -C\tau s^2(\alpha^2+|\lambda|)  \int_{ \mathbb B_{1/2}^+}\psi \beta w^2\,dy.
\label{LL4}
\end{align}

We want to find out a lower estimate for $L_1$ to include the gradient terms. It follows from integration by parts and Cauchy-Schwartz inequality  that
\begin{align}
L_1&=\tau s^2 \int_{ \mathbb B_{1/2}^+}\psi \beta a_{ij} D_i wD_j w\, dy+\tau s^2  \int_{ \mathbb B_{1/2}^+} D_i (a_{ij} \psi \beta) w D_j w \,dy \nonumber \\ &-\tau s^2  \int_{\partial \mathbb B_{1/2}^+} \psi \beta w  a_{ij} D_i w \nu_j\, dS \nonumber \\
&\geq \tau s^2 \int_{ \mathbb B_{1/2}^+}\psi \beta | a_{ij} D_i wD_j w| \, dy-C\tau s^3 \int_{ \mathbb B_{1/2}^+}\psi |\nabla w| |w|\,dy\nonumber \\& -\tau s\int_{\partial \mathbb B_{1/2}^+}\psi\beta w a_{ij} D_j w\nu_j\, dS\nonumber \\
&\geq  \frac{9}{10}\tau s^2 \int_{ \mathbb B_{1/2}^+}\psi \beta a_{ij}D_i w D_j w \, dy-C\tau s^{10} \int_{ \mathbb B_{1/2}^+}\psi^2 w^2 \, dy-C\tau s^2 \int_{\partial \mathbb B_{1/2}^+}\psi w^2\, dS \nonumber \\
&-C\tau s^2 \int_{\partial \mathbb B_{1/2}^+}\psi |\nabla w|^2\, dS.
\label{LL1}
\end{align}
Taking (\ref{product}), (\ref{LL2}), (\ref{LL3}), (\ref{LL4}), (\ref{LL1}),  and $\tau> C_s(|\alpha|+\sqrt{|\lambda|})$ into account gives that
\begin{align}
\langle P_1 w, \ \frac{5\tau s^2}{2}\psi\beta(y) w\rangle &\geq \frac{9\tau s^2}{4} \int_{ \mathbb B_{1/2}^+}\psi \beta a_{ij}D_i w D_j w \, dy- \frac{5\tau^3 s^4}{2}\int_{ \mathbb B_{1/2}^+}\psi^3 \beta^2 w^2\,dy \nonumber \\
& -C\tau s^2 \int_{\partial \mathbb B_{1/2}^+}\psi w^2\, dS-C\tau s^2 \int_{\partial \mathbb B_{1/2}^+}\psi |\nabla w|^2\, dS.
\label{LLLii}
\end{align}

Since
\begin{align}
\|P_1 w\|^2 +\frac{25}{4}\|\tau s^2\psi\beta w\|^2 \geq 2 \langle P_1 w,  \frac{5\tau s^2}{2}\psi\beta w\rangle,
\end{align}
from (\ref{keyinner}), we obtain that
\begin{align}
\|P_\tau w\|^2 +\frac{25}{4}\|\tau s^2\psi\beta w\|^2 &\geq 2 \langle P_1 w,  \frac{5\tau s^2}{2}\tau s^2\psi\beta w\rangle+ 2 \langle P_1 w,  P_2 w\rangle \nonumber \\ &-\|\tau\psi a(y,s)w\|^2.
\end{align}
From the expression of $a(y, s)$ in (\ref{axs}), we can absorb $\|\tau\psi a(y,s)w\|^2$ into the inner product $\langle P_1 w, \  P_2 w\rangle$
by the dominating term $\tau^3 s^4 \int_{\mathbb B_{1/2}^+}\psi^3\beta^2 w^2 \, dy$ in (\ref{inner}).
We can absorb $\|\tau s^2\psi\beta w\|^2$ into the inner product $\langle P_1 w, \  P_2 w\rangle$ as well. Thanks to (\ref{keyinner}), (\ref{inner}) and (\ref{LLLii}), using the assumption that $\tau >  C_s(|\alpha|+\sqrt{|\lambda|})$ and $s$ is a fixed large constant,
we arrive at
\begin{align}
\|P_\tau w\|^2&+ \tau^3 s^3  \int_{\partial \mathbb B_{1/2}^+}\psi^3   w^2 \, dS+ \tau s^2
\int_{\partial \mathbb B_{1/2}^+}\psi |\nabla w|^2 \, dS \nonumber \\
 &\geq C\tau^3 s^4 \int_{\mathbb B_{1/2}^+}\psi^3  w^2 \, dy + C\tau s^2\int_{ \mathbb B_{1/2}^+} \psi|\nabla w|^2\,dy.
\end{align}
Recall (\ref{wvv}) and the operator $P_0$. We derive the following Carleman estimates for $v$ as
\begin{align}
\| e^{\tau \psi} P_0v \|^2&+ \tau^3 s^3  \int_{\partial \mathbb B_{1/2}^+}\psi^3 e^{2\tau \psi}   v^2 \, dS+ \tau s^2
\int_{\partial \mathbb B_{1/2}^+}\psi  e^{2\tau \psi}|\nabla v|^2 \, dS \nonumber \\
 &\geq C\tau^3 s^4 \int_{\mathbb B_{1/2}^+}\psi^3  e^{2\tau \psi} v^2 \, dy + C\tau s^2\int_{ \mathbb B_{1/2}^+} \psi e^{2\tau \psi}|\nabla v|^2\,dy.
\end{align}
Thus, we arrive at the conclusion in the proposition.
\end{proof}

\section{Appendix}
In this section, we provide the proof of  Lemma \ref{lemm1} in section 3.

\begin{proof}[Proof of Lemma \ref{lemm1}]
Without loss of generality, we may assume $x_0$ as the origin.
We select $R$ satisfying $0<R<\frac{r_0}{6}$ with $r_0$ as in the proposition \ref{pro1}.
Set the annulus $A_{R_1, R_2}=\{ y\in \mathbb B_\delta; R_1\leq r(y)\leq R_2\}$. Thus, $\|v\|_{R_1, R_2}$ is the $L^2$ norm of $v$  in the annulus  $A_{R_1, R_2}$.
We introduce a smooth cut-off function $\eta(r)\in C^\infty_0(\mathbb B_{3R})$ with $0<\eta(r)<1$ satisfying the following properties:
\begin{itemize}
\item $\eta(r)=0$ \ \ \mbox{if} \ $r(y)<\frac{R}{4}$ \ \mbox{or} \  $r(y)>\frac{5R}{2}$,\medskip
\item $\eta(r)=1$ \ \ \mbox{if} \ $\frac{3R}{4}<r(y)<\frac{9R}{4}$,\medskip
\item $|\nabla \eta|\leq \frac{C}{R},$\medskip
\item $|\nabla^2\eta|\leq \frac{C}{R^2}.$
\end{itemize}
Due to the definition of $\eta$, the function $\eta \bar u$ has compact support in the annulus $A_{\frac{R}{4}, \frac{5R}{2}}$.  Applying the Carleman estimates (\ref{carl}) with $v$ replaced by $\eta \bar u$ and taking it consideration that $\bar u$ is the solution for the elliptic equations (\ref{target}) yields that
\begin{align}
\tau \| e^{\tau \psi} \bar u  \|&\leq \|r^2  e^{\tau \psi}\big(\triangle(\eta \bar u )+\hat{b}(y)\cdot\nabla (\eta \bar u )+ \hat{c}(y) \eta \bar u \big)\| \nonumber \\
&\leq C\| r^2 e^{\tau \psi}(\triangle \eta \bar u +2\nabla\eta\cdot \nabla \bar u +\hat{b}\cdot \nabla \eta \bar u )\|.
\end{align}
Notice that the parameter $\tau\geq 1$. From the properties of $\eta$, it follows that
\begin{align*}
\| e^{\tau \psi} \bar u \|_{\frac{3R}{4}, \frac{9R}{4}}&\leq C (\| e^{\tau \psi} \bar u \|_{\frac{R}{4}, \frac{3R}{4}}+\| e^{\tau \psi} \bar u \|_{\frac{9R}{4}, \frac{5R}{2}} ) \\
&+ C( R\| e^{\tau \psi} \nabla \bar u \|_{\frac{R}{4}, \frac{3R}{4}}+ R\| e^{\tau \psi} \nabla \bar u \|_{\frac{9R}{4}, \frac{5R}{2}}) \\
&+ C (|\alpha|+1)R (\| e^{\tau \psi} \bar u \|_{\frac{R}{4}, \frac{3R}{4}}+\| e^{\tau \psi}\bar u \|_{\frac{9R}{4}, \frac{5R}{2}} ).
\end{align*}
Since $R\leq 1$, we obtain that
\begin{align*}
\| e^{\tau \psi} \bar u \|_{\frac{3R}{4}, \frac{9R}{4}}&\leq C(|\alpha|+1) (\| e^{\tau \psi} \bar u \|_{\frac{R}{4}, \frac{3R}{4}}+\| e^{\tau \psi} \bar u \|_{\frac{9R}{4}, \frac{5R}{2}} ) \\
&+ C( R\| e^{\tau \psi} \nabla \bar u \|_{\frac{R}{4}, \frac{3R}{4}}+ R\| e^{\tau \psi} \nabla \bar u \|_{\frac{9R}{4}, \frac{5R}{2}}).
\end{align*}

Recall the weight function $\psi(r)=-\ln r- \ln (\ln r)^2 $. We see that $\psi(r)$ is radial and decreasing. Thus, we can deduce that
\begin{align}
e^{\tau \psi(2R)}\|  \bar u \|_{\frac{3R}{4}, 2R}&\leq C(|\alpha|+1) (e^{\tau \psi(\frac{R}{4})} \|  \bar u \|_{\frac{R}{4}, \frac{3R}{4}}+e^{\tau \psi(\frac{9R}{4})} \| \bar u \|_{\frac{9R}{4}, \frac{5R}{2}} ) \nonumber \\
&+ C( Re^{\tau \psi(\frac{R}{4})}\| \nabla \bar u \|_{\frac{R}{4}, \frac{3R}{4}}+ Re^{\tau \psi(\frac{9R}{4})}\|  \nabla \bar u \|_{\frac{9R}{4},
\frac{5R}{2}}). \label{mdd}
\end{align}
For the equation (\ref{target}), it is known that the Caccioppoli type inequality
\begin{equation}
\|\nabla \bar u \|_{(1-a)r}\leq \frac{C(|\alpha|+\sqrt{|\lambda|})}{r}\| \bar u \|_{r}
\label{Cacc}
\end{equation}
holds  with any $0<a<1$. Applying such inequality gives that
\begin{align}
R \|  \nabla \bar u \|_{\frac{R}{4}, \frac{3R}{4}}\leq C (|\alpha|+\sqrt{|\lambda|}) \| \bar u \|_{R}.
\label{give1}
\end{align}
Using the same strategy implies that
\begin{align}
R \|  \nabla \bar u \|_{\frac{9R}{4}, \frac{5R}{2}}\leq C(|\alpha|+\sqrt{|\lambda|}) \| \bar u \|_{3R}.
\label{give2}
\end{align}
Substituting (\ref{give1}) and (\ref{give2}) into (\ref{mdd}) gives that
\begin{align}
\|\bar u \|_{\frac{3R}{4}, 2R}\leq C (|\alpha|+\sqrt{|\lambda|}) \big( e^{\tau(\psi(\frac{R}{4})-\psi(2R))} \|\bar u \|_R+  e^{\tau(\psi(\frac{9R}{4})-\psi(2R))} \|\bar u \|_{3R}\big).
\label{add}
\end{align}
We introduce two parameters
$$ \beta^1_R=\psi(\frac{R}{4})-\psi(2R),$$
$$ \beta^2_R=\psi(2R)-\psi(\frac{9R}{4}).$$
Due to the explicit form of $\psi $, we can check that
$$ 0<\beta^{-1}_1<\beta^1_R<\beta_1 \quad \mbox{and} \quad 0<\beta_2<\beta^2_R<\beta^{-1}_2 $$
for some $\beta_1$ and $\beta_2$ independent of $R$.
Adding $\|\bar u \|_{\frac{3R}{4}}$ to both sides of the inequality (\ref{add}) and considering that $\psi(\frac{R}{4})-\psi(2R)>0$, we get that
\begin{equation}
\|\bar u \|_{2R}\leq C(|\alpha|+\sqrt{|\lambda|})\big( e^{\tau\beta_1}\|\bar u \|_R+  e^{-\tau\beta_2}\|\bar u \|_{3R}     \big).
\end{equation}
In order to move the second term on the right hand side of the last inequality to the left hand side, we choose $\tau$ such that
\begin{equation} C(|\alpha|+\sqrt{|\lambda|})e^{-\tau\beta_2}\|\bar u \|_{3R}\leq \frac{1}{2}\|\bar u \|_{2R}.
\label{have}  \end{equation}
To have (\ref{have}), it is enough to require
$$\tau\geq \frac{1}{\beta_2} \ln \frac{2C(|\alpha|+\sqrt{|\lambda|})\|\bar u \|_{3R}}{\|\bar u \|_{2R} }.   $$
Because of such choice of $\tau$, we obtain that
\begin{equation}
\|\bar u \|_{2R}\leq C(|\alpha|+\sqrt{|\lambda|}) e^{\tau\beta_1}\|\bar u \|_R.
\label{substitute2}
\end{equation}
Recall the assumption that the parameter $\tau>C(\sqrt{|\lambda|}+|\alpha|)$ in Carleman estimates (\ref{carl}). We choose
$$ \tau=C(|\alpha|+\sqrt{|\lambda|})+\frac{1}{\beta_2} \ln \frac{2C(|\alpha|+\sqrt{|\lambda|})\|\bar u\|_{3R}}{\|\bar u\|_{2R} }. $$
Substituting such $\tau$ in (\ref{substitute2}) yields that
\begin{align}
\|\bar u \|_{2R}^{\frac{\beta_2+\beta_1}{\beta_2}} \leq e^{ C (|\alpha|+\sqrt{|\lambda|}) }\|\bar u \|_{3R}^{\frac{\beta_1}{\beta_2}} \|\bar u \|_R.
\end{align}
Raising exponent $\frac{\beta_2}{\beta_2+\beta_1}$ to both sides of last inequality yields that
\begin{align}
\|\bar u \|_{2R} \leq e^{C (|\alpha|+\sqrt{|\lambda|})}\|\bar u \|_{3R}^{\frac{\beta_1}{\beta_1+\beta_2}} \|\bar u \|_R^{\frac{\beta_2}{\beta_1+\beta_2}},
\end{align}
where we have used again the fact that $\beta_1$, $\beta_2$ independent of $R$. Let $$\beta={\frac{\beta_2}{\beta_1+\beta_2}}.$$  Therefore,  the quantitative three-ball inequality in Lemma \ref{lemm1} is obtained.
\end{proof}

\noindent{\bf Acknowledgement.} The author thanks Professor Steve Zelditch for bringing the reference \cite{TZ} to our attentions and helpful discussions.


\begin{thebibliography}{CL}

\bibitem[ARRV]{ARRV} G. Alessandrini, L. Rondi, E. Rosset and S. Vessella,
The stability for the Cauchy problem for elliptic equations.
Inverse Problems,  25(2009), no. 12, 123004, 47 pp.

\bibitem[AKS]{AKS} N. Aronszajn, A. Krzywicki and J. Szarski,  A unique continuation theorem for exterior differential forms on Riemannian manifolds, Ark. Mat.,  4(1962), 417--453.

\bibitem[BC]{BC}L. Bakri and J.B.  Casteras,  Quantitative uniqueness for Schr\"{o}dinger operator with regular
potentials, Math. Methods Appl. Sci., 37(2014), 1992--2008.

\bibitem[BL]{BL}K. Bellova and F.-H. Lin, Nodal sets of Steklov eigenfunctions,
Calc. Var. $\&$ PDE, 54(2015), 2239--2268.

\bibitem[Br]{Br} J. Br\"uning, \"Uber Knoten won Eigenfunktionen des
Laplace-Beltrami-Operators, Math. Z., 158(1978), 15--21.


 \bibitem[CM]{CM} T.H. Colding and W. P. Minicozzi II, Lower bounds for nodal sets
of eigenfunctions, Comm. Math. Phys., 306(2011), 777--784.

\bibitem[DK]{DK} D. Daners and J. Kennedy,
On the asymptotic behaviour of the eigenvalues of a Robin problem,
Differential Integral Equations, 23(2010), no. 7-8, 659--669.

\bibitem[DF]{DF}
H. Donnelly and C. Fefferman,
 Nodal sets of eigenfunctions on {R}iemannian manifolds,
 {Invent. Math.}, 93(1988), no. 1, 161--183.

 \bibitem[DF1]{DF1}
 H. Donnelly and C. Fefferman, Nodal sets for eigenfunctions of the Laplacian on surfaces, J. Amer.
Math. Soc., 3(1990), no. 2, 333--353.

\bibitem[DF2]{DF2} H. Donnelly and C. Fefferman, Nodal sets of eigenfunctions:
Riemannian manifolds with boundary, in: Analysis, Et Cetera,
Academic Press, Boston, MA, 1990, 251--262.

\bibitem[D]{D} R-T Dong, Nodal sets of eigenfunctions on Riemann surfaces, J. Differential Geom., 36(1992), 493--506.


 \bibitem[ES]{ES} A. Elbert and P. Siafarikas, On the zeros of $aC_{\nu}(x)+x C_{\nu}(x)$, where $C_{\nu}(x)$ is a cylinder function, J. Math. Anal. Appl., 164(1992), no.1, 21--33.

\bibitem[GL]{GL} N. Garofalo and F.-H. Lin, Monotonicity properties of
variational integrals, $A_p$ weights and unique continuation,
{Indiana Univ. Math.}, 35(1986), 245--268.

\bibitem[GR]{GR} B. Georgiev and G. Roy-Fortin, Polynomial upper bound on interior Steklov nodal sets, arXiv:1704.04484.

\bibitem[Han]{Han} Q. Han, Nodal sets of harmonic functions, Pure Appl. Math. Q. 3(2007), no.3, part 2, 647-688. 

\bibitem[HL]{HL} Q. Han and F.-H. Lin, Nodal sets of solutions of Elliptic
Differential Equations, book in preparation (online at
http://www.nd.edu/qhan/nodal.pdf).

\bibitem[HLu]{HLu} X. Han and G. Lu, A geometric covering lemma and nodal sets of eigenfunctions, Math. Res. Lett., 18(2011), no. 2, 337-352.


\bibitem[HS]{HS} R. Hardt and L. Simon, Nodal sets for solutions of
ellipitc equations, J. Differential Geom., 30(1989), 505--522.

\bibitem[HSo]{HSo} H. Hezari and C.D. Sogge, A natural lower bound for the
size of nodal sets, Anal. PDE., 5(2012), no. 5, 1133--1137.

\bibitem[H]{H}L. H\"ormander, The analysis of linear partial differential operators I. Distribution theory and Fourier analysis, Reprint of the second (1990) edition, Classics in Mathematics. Springer-Verlag, Berlin, 2003.

\bibitem[JL]{JL} D. Jerison and G. Lebeau, Nodal sets of sums of eigenfunctins, Harmonic analysis and partial differential equations (Chicago, IL, 1996), 223--239,
Chicago Lectures in Math., Uniw. Chicago Press, Chicago, IL, 1999.



\bibitem[K]{K}
C.  Kenig,
 Some recent applications of unique continuation,
 In  Recent developments in nonlinear partial differential
  equations,  25--56, Contemp. Math, 439, Amer. Math.
  Soc., Providence, RI, 2007.

    \bibitem[LR]{LR} G. Lebeau and L. Robbiano, Contr\^ole exacte de l'\'equation de la chaleur, Comm. Partial Differential Equations, 20(1995), 335--356.

\bibitem[LL]{LL}J. Le Rousseau and G. Lebeau, On Carleman estimates for elliptic and parabolic operators. Applications to unique continuation and control of parabolic equations, ESAIM Control Optim. Calc. Var.,  18(2012), no.3, 712--747.

\bibitem[LZ]{LZ} M. Li and X. Zhou, Min-max theory for free boundary minimal hypersurfaces I: regularity theory, arXiv:1611.02612.

\bibitem[Lin]{Lin}F.-H. Lin, Nodal sets of solutions of elliptic
equations of elliptic and parabolic equations, Comm. Pure Appl Math.,
44(1991), 287--308.

\bibitem[Lo]{Lo} A. Logunov, Nodal sets of Laplace eigenfunctions: polynomial upper estimates of the Hausdorff measure, Annals of Mathematics, 187(2018), 221--239.

\bibitem[Lo1]{Lo1} A. Logunov, Nodal sets of Laplace eigenfunctions: proof of Nadirashvili's conjecture and of the lower bound in Yau's conjecture, Annals of Mathematics,  187(2018), 241--262.

\bibitem[LM]{LM} A. Logunov and E. Malinnikova, Nodal sets of Laplace eigenfunctions: estimates of the Hausdorff measure in dimension two and three, 	50 years with Hardy spaces, 333-344, Oper. Theory Adv. Appl., 261, Birkhäuser/Springer, Cham, 2018.
    
\bibitem[Lu]{Lu} G. Lu, Covering lemmas and an application to nodal geometry on Riemannian manifolds, Proc. Amer. Math. Soc, 117(1993), no.4, 971-978.

\bibitem[M]{M} D. Mangoubi, A remark on recent lower bounds for nodal sets, Comm. Partial Differential Equations, 36(2011), no. 12, 2208--2212.

\bibitem[MN]{MN} C.B. Morrey and L. Nirenberg, On the analyticity of the solutions of linear elliptic systems of partial differential equations, 10(1957) 271-290.

\bibitem[O]{O} F. Olver, The asymptotic expansion of Bessel function of large
order, Philos. Trans. Roy. Soc. London Ser. A,  247(1954), 328--368.

\bibitem[PST]{PST} I. Polterovich, D. Sher and J. Toth,
 Nodal length of Steklov eigenfunctions on real-analytic
Riemannian surfaces, to appear in  J. Reine Angew. Math.

\bibitem[R]{R} A. R\"uland,  quantitative unique continuation properties of fractional Schrödinger equations: doubling, vanishing order and nodal domain estimates, Trans. Amer. Math. Soc., 369(2017), no.4, 2311-2362.

    \bibitem[SWZ]{SWZ} C.D. Sogge, X. Wang and J. Zhu, Lower bounds for interior nodal sets of Steklov eigenfunctions,
 Proc. Amer. Math. Soc., 144(2016), no. 11, 4715--4722.

\bibitem[SZ]{SZ} C.D. Sogge and S. Zelditch, Lower bounds on the
Hausdorff measure of nodal sets, Math. Res. Lett., 18(2011), 25--37.


\bibitem[S]{S} S. Steinerberger, Lower bounds on nodal sets of eigenfunctions via the heat flow, Comm. Partial Differential Equations, 39(2014), no. 12, 2240--2261.

\bibitem[TZ]{TZ}J. Toth and S. Zelditch,
Counting nodal lines which touch the boundary of an analytic domain,
J. Differential Geom., 81(2009), no.3, 649-686.

\bibitem[T]{T} T. Treves, Basic linear partial differential
equations, Academic Press, N.Y., 1975.


\bibitem[Y]{Y} S.T. Yau, Problem section, seminar on differential geometry, Annals of Mathematical Studies 102, Princeton, 1982, 669--706.

\bibitem[WZ]{WZ} X. Wang and J. Zhu,  A lower bound for the nodal sets of Steklov
eigenfunctions,  Math. Res. Lett., 22(2015), no.4, 1243--1253.

\bibitem[Z]{Z} S. Zelditch, Measure of nodal sets of analytic
steklov eigenfunctions, Math. Res. Lett., 22(2015), no.6, 1821--1842.

\bibitem[Zh]{Zh}J. Zhu, Doubling property and vanishing order of Steklov
eigenfunctions, Comm. Partial Differential Equations, 40(2015), no.
8, 1498-1520.

\bibitem[Zh1]{Zh1} J. Zhu, Interior nodal sets of Steklov eigenfunctions on surfaces,
 Anal. PDE, 9(2016), no. 4, 859--880.

\bibitem[Zh2]{Zh2}J. Zhu, Geometry and interior nodal sets of Steklov
  eigenfunctions, arXiv:1510.07300.


\end{thebibliography}
\end{document}